\documentclass[reqno]{amsart}

\usepackage{amsmath}
\usepackage{hyperref}
\hypersetup
{
    colorlinks,
    citecolor=blue,
    filecolor=red,
    linkcolor=blue,
    urlcolor=black
}
\usepackage{amssymb}
\usepackage{amsthm}
\usepackage{extarrows}
\usepackage{enumitem}
\usepackage{mathrsfs}
\usepackage{color}
\usepackage
[
a4paper,
vmargin=3cm,
hmargin=2.5cm 
]{geometry}

\newcommand\R{\mathbb R}
\newcommand\Z{\mathbb Z}
\newcommand\T{\mathbb T}
\newcommand\N{\mathbb N}
\newcommand\E{\mathbb E}
\newcommand\e{\varepsilon}
\newcommand\p{\mathbb P}

\newcommand\W{\mathcal W}
\newcommand\s{\mathcal S}
\newcommand\pr{\mathcal Pr}
\newcommand\B{\mathcal B}

\theoremstyle{plain}
\numberwithin{equation}{section}
\newtheorem{Theorem}{Theorem}[section]
\newtheorem{Proposition}{Proposition}[section]
\newtheorem{Lemma}{Lemma}[section]
\newtheorem{Definition}{Definition}[section]
\theoremstyle{definition}
\newtheorem{Remark}{Remark}[section]

\newtheorem{Assumption}{Assumption}

\begin{document}

\title[Stochastic Camassa-Holm Type Equation]{On a stochastic Camassa-Holm type equation with higher order nonlinearities}

\author[C. Rohde]{Christian Rohde}
\address{Institut f\"{u}r Angewandte Analysis und Numerische Simulation, Universit\"{a}t Stuttgart, Pfaffenwaldring 57, 70569 Stuttgart, Germany}
\email{christian.rohde@mathematik.uni-stuttgart.de}
\thanks{C. R. acknowledges funding  by Deutsche Forschungsgemeinschaft
	(DFG, German Research Foundation) under
	Germany's Excellence Strategy --EXC 2075--390740016}

\author[H. Tang]{Hao Tang}
\address{Institut f\"{u}r Angewandte Analysis und Numerische Simulation, Universit\"{a}t Stuttgart, Pfaffenwaldring 57, 70569 Stuttgart, Germany}
\email[Corresponding author]{Hao.Tang@mathematik.uni-stuttgart.de}
\thanks{H. T. is supported by the Alexander von Humboldt Foundation.}






\begin{abstract}
	The subject of this paper is a generalized Camassa-Holm equation under random perturbation. We first establish  local existence and uniqueness results as well as  blow-up criteria for pathwise solutions in the Sobolev spaces $H^s$ with $s>3/2$. Then we analyze how noise affects the dependence
	of solutions  on initial data. Even though the noise has some already known regularization effects, much less is known concerning  the dependence on initial data. 
	As a new concept we  introduce the notion of stability of exiting times and construct an example showing that  multiplicative noise (in It\^{o} sense) cannot  improve the stability of the exiting time, and simultaneously improve the continuity of the dependence on initial data. Finally, we  obtain global existence theorems  and estimate  associated probabilities. 
	
\end{abstract}

\maketitle

%

\section{Introduction}
We consider a stochastic version of the generalized Camassa-Holm equation. Let $t$  denote the time variable and  let $x$ 
be the one-dimensional space variable.   The equation 
is given for $k\in \N$ by
\begin{equation}\label{gCH random dissipation}
u_t-u_{xxt}+(k+2)u^ku_x-(1-\partial^2_{xx})h(t,u)\dot{\mathcal W}=(k+1)u^{k-1}u_xu_{xx}+u^ku_{xxx}. 
\end{equation} 
In \eqref{gCH random dissipation}, $h:\R^+\times \R \to \R$ is some nonlinear function 
and $\dot{\mathcal W}$  is a cylindrical Wiener process.
We will consider \eqref{gCH random dissipation}
on the torus, i.e.  $x\in \T=\R/2\pi\Z$.  

For  $h= 0$ and $k=1$, equation \eqref{gCH random dissipation} reduces to the deterministic 
Camassa-Holm (CH) equation given by
\begin{eqnarray}
u_{t}-u_{xxt}+3uu_{x}=2u_{x}u_{xx}+uu_{xxx}.\label{CH}
\end{eqnarray}
Fokas \& Fuchssteiner \cite{Fuchssteiner-Fokas-1981-PhyD} introduced \eqref{CH}  when studying completely integrable 
generalizations of the Korteweg-de-Vries (KdV) equation with bi-Hamiltonian structure whereas  Camassa \& Holm~\cite{Camassa-Holm-1993-PRL} 
proposed \eqref{CH} to describe the unidirectional propagation of shallow water waves over a flat bottom. 
Since then, the CH equation \eqref{CH} has been studied intensively, and we cannot even attempt to survey the vast research history
here. For the paper at hand it is important to mention  the wave-breaking phenomenon which illustrates possible loss of regularity as a fundamental  mechanism  in the  CH equations.  In contrast to smooth 
soliton solutions to the KdV equation~\cite{Kenig-Ponce-Vega-1993-CPAM},  solutions to the CH equation  remain  indeed bounded but  
their slope can become unbounded in finite time, cf.~\cite{Constantin-Escher-1998-Acta} and  related work in  \cite{Constantin-2000-JNS,Constantin-Escher-1998-CPAM,Mckean-1998-AJM}. Moreover, for a smooth initial profile, it is possible to
 predict  exactly (by establishing a necessary and sufficient condition) whether wave-breaking occurs  for solutions to the Cauchy problem for \eqref{CH} \cite{Constantin-Escher-1998-Acta,Mckean-1998-AJM}.  The other essential feature of the CH equation
 is the occurrence of traveling waves with a peak at their crest, exactly like that the governing equations for water waves admit the so-called Stokes waves of the greatest height, see \cite{Constantin-2006-Inventiones,Constantin-Escher-2007-BAMS,Constantin-Escher-2011-Annals}. 
  Bressan \& Constantin 
  proved the existence of  dissipative  as well as conservative solutions in \cite{Bressan-Constantin-2007-ARMA,Bressan-Constantin-2007-AA}.
   Later, Holden \& Raynaud \cite{Holden-Raynaud-2007-CPDE,Holden-Raynaud-2009-DCDS} also
obtained  global conservative and dissipative solutions using Lagrangian transport ideas.

When $h= 0$ and $k=2$, equation \eqref{gCH random dissipation}  forms the cubic equation  
\begin{equation}\label{Novikov}
u_t-u_{xxt}+4u^2u_x=3uu_xu_{xx}+u^2u_{xxx},
\end{equation} 
which has been derived by Novikov in \cite{Novikov-2009-JPA}. It has been  proven that \eqref{Novikov} possesses a bi-Hamiltonian structure 
with an  infinite sequence of conserved quantities and \eqref{Novikov}  admits peaked solutions  of  explicit form 
$u(t, x)=\pm \sqrt{c} e^{-|x-ct|}$ with $c>0$ \cite{Geng-Xue-2009-Nonlinearity}, as well as   multipeakon solutions  with explicit formulas \cite{Hone-etal-2009-DPDE}.
For the study of other deterministic instances of \eqref{gCH random dissipation}  
we refer to \cite{Chen-Li-Yan-2015-DCDS,Himonas-Holliman-2014-ADE,Zhao-Li-Yan-2014-DCDS,Zhou-Mu-2013-JNS}.

Equations like \eqref{CH} are naturally embedded  in high-order descriptions with energy-dissipative evolution. A weakly dissipative evolution  is given  for some parameter $\lambda >0$   by 
\begin{equation}\label{dissipativeCH}
u_t-u_{xxt}+(k+2)u^ku_x- (1-\partial^2_{xx}) (\lambda u)=(k+1)u^{k-1}u_xu_{xx}+u^ku_{xxx},\qquad x\in \T, \, t>0. 
\end{equation} 
 For instance, the weakly dissipative CH equation, i.e. \eqref{dissipativeCH} for $k=1$, has been studied in 
 \cite{Lenells-Wunsch-2013-JDE,Wu-Yin-2009-JDE},
For the Novikov equation  \eqref{Novikov} with the same weakly dissipative 
term $(1-\partial^2_{xx})(\lambda u)$, we  can also refer to \cite{Lenells-Wunsch-2013-JDE}.
  
In this work, we assume that the energy exchange mechanisms are connected with randomness to account for external stochastic influence. 
We are interested in  the case   where   in the term $(1-\partial^2_{xx})(\lambda u)$  the deterministic parameter $\lambda$ is  substituted by a formal cylindrical Wiener process (cf.~Section \ref{Preliminaries and Main Results} for precise definitions), and where  the  previously linear dependency on $u$ is  replaced by a non-autonomous and  nonlinear term $
h(t,u)$.  Thus,  we  consider the Cauchy problem for \eqref{gCH random dissipation} on the 
torus $\T$, 
with  random initial data $u_0=u_0(\omega,x)$.  Applying the operator $(1-\partial_{xx}^2)^{-1}$ to  \eqref{gCH random dissipation}, we reformulate  the problem as the stochastic evolution 
	\begin{equation} \label{SGCH problem}
\left\{
\begin{aligned}
{\rm d}u+\left[u^k\partial_xu+F(u)\right]{\rm d}t&=h(t,u){\rm d}\W,\quad x\in\T, \ t>0,\  k\geq 1,\\
u(\omega,0,x)&=u_0(\omega,x),\quad x\in\T,
\end{aligned} 
\right.
\end{equation}
%
%
%
%
with $F(u)=F_1(u)+F_2(u)+F_3(u)$ and
\begin{equation}\label{F decomposition}
\begin{aligned}
&F_1(u)=(1-\partial_{xx}^2)^{-1}\partial_x\left(u^{k+1}\right),\\
&F_2(u)=\frac{2k-1}{2}(1-\partial_{xx}^2)^{-1}\partial_x\left(u^{k-1}u_x^2\right),\\
&F_3(u)=\frac{k-1}{2}(1-\partial_{xx}^2)^{-1}\left(u^{k-2}u_x^3\right).
\end{aligned} 
\end{equation}
Here we remark that the operator $(1-\partial_{xx}^2)^{-1}$ in $F(\cdot)$ is understood as
\begin{align}\label{Helmboltz operator}
\left[(1-\partial_{xx}^2)^{-1}f\right](x)=[G_{\T}*f](x) \ \text{for}\ f\in L^2(\T) \ \text{with}\  G_{\T}=\frac{\cosh(x-2\pi\left[\frac{x}{2\pi}\right]-\pi)}
{2\sinh(\pi)},
\end{align}
where the convolution is on $\T$ only, and the symbol $[x]$ stands for the integer part of $x$.

The first objective  of this paper is to analyze 
the	   local  existence and  uniqueness  of  pathwise solutions   as well as  blow-up criteria for  problem    \eqref{SGCH problem}
 with nonlinear  multiplicative noise (see Theorem \ref{Local pathwise solution}). 
%
We note  that  for the CH equation with additive noise, existence and uniqueness has been obtained in
\cite{Chen-Gao-Guo-2012-JDE}. For the  stochastic modified CH equation with linear multiplicative noise, we refer to \cite{Chen-Gao-2016-PA}. When the multiplicative noise is given by a one-dimensional Wiener process with $H^s$ diffusion coefficient, the stochastic CH equation was considered in second author's work \cite{Tang-2018-SIMA}. In this paper, we consider the local  existence and  uniqueness  of  pathwise solutions to  \eqref{SGCH problem} as well as  blow-up criteria for more general noise given by a cylindrical Wiener process with nonlinear coefficient.

The second major objective of the paper is to investigate whether stochastic perturbations can improve the dependence on initial data in the Cauchy problem \eqref{SGCH problem}.
We notice that various instances of  transport-type stochastic evolution laws  have been
studied with respect to the  effect of  noise on the regularity of their solutions. 
For example, we refer to \cite{Flandoli-Gubinelli-Priola-2010-Invention,Mollinedo-Olivera-2017-BBMS,Neves-Olivera-2015-NODEA,Fedrizzi-Neves-Olivera-2018-ANSP} for linear stochastic transport equations and to what concerns nonlinear stochastic conservation laws, we refer to \cite{Lions-Perthame-Souganidis-2013-SPDE,Lions-Perthame-Souganidis-2014-SPDE,Gess-Souganidis-2017-CPAM}.  In   \cite{Kim-2010-JFA,GlattHoltz-Vicol-2014-AP,Rockner-Zhu-Zhu-2014-SPTA,Tang-2018-SIMA} the dissipation of energy caused by 
linear multiplicative noise has been analyzed.  

In contrast to these works we  focus in this paper 
on the initial-data dependence in the Cauchy problem \eqref{SGCH problem}. Actually,
much less is known on the noise effect with respect to  the  dependence on initial data. But
 the question whether (and how)  noise can affect  initial-data dependence is interesting. Formally speaking, 
 regularization produced by noise may be  seen like  the  deterministic regularization effect  induced by adding a Laplacian. However,  if one would  add a real Laplacian to the governing equations, then by using  techniques from the theory for  semilinear parabolic equations, the dependence on initial data in some cases turns out to be more regular than being just continuous. For example, for the deterministic incompressible Euler equations, the dependence on initial data cannot be better than continuous \cite{Himonas-Misiolek-2010-CMP}, but for the deterministic incompressible Navier-Stokes equations with sufficiently large viscosity, it is at least Lipschitz continuous in sufficiently high Sobolev norms, see pp. 79--81 in \cite{Henry-1981-book}. This  motivates us  to study whether (and how)  noise can affect  initial-data dependence.  Actually, to our knowledge, there are very few results in this direction. We can only refer to \cite{Tang-2020-Arxiv} for the stochastic Euler-Poincar\'{e} equation with respect to this problem. However, in this work, we have higher order nonlinearities, which requires more technical estimates.
 To analyze initial-data dependence, we introduce the  concept of 
 the stability of the exiting time (see Definition \ref{Stability on exiting time} below), cf. \cite{Tang-2020-Arxiv}. Then we show under some conditions on $h(t,u)$, that the multiplicative noise (in the It\^{o} sense) cannot improve the stability of the exiting time, \textit{and}, at the same time, improve the continuity of the map $u_0\mapsto u$ (see Theorem \ref{weak instability}).  
 It is worth noting that in deterministic cases, the issue of the optimal dependence of solutions (for example, the solution map is continuous but not uniformly continuous) to various nonlinear dispersive and integrable equations has been the subject of many papers. One of the first results of this type  dates back at least as far as 
 to Kato \cite{Kato-1975-ARMA}. Indeed, Kato \cite{Kato-1975-ARMA} proved that the solution map $u_{0}\mapsto u$ in $H^s(\T)$ ($s>3/2$), given by the inviscid Burgers equation,  is not  H\"{o}lder continuous regardless of the H\"{o}lder exponent. Since then different techniques have been successfully applied to various nonlinear dispersive and integrable equations, see \cite{Kenig-Ponce-Vega-2001-Duke,Koch-Tzvetkov-2005-IMRN,Bona-Tzvetkov-2009-DCDS} for example. Particularly, for the incompressible Euler equation, we refer to \cite{Himonas-Misiolek-2010-CMP,Tang-Liu-2014-JMP}, and for CH type equations, we refer to \cite{Himonas-Kenig-2009-DIE,Himonas-Kenig-Misiolek-2010-CPDE,Tang-Liu-2015-ZAMP,Tang-Shi-Liu-2015-MM,Tang-Zhao-Liu-2014-AA} and the references therein.

We conclude the paper with a discussion of time-global well-posedness for  \eqref{SGCH problem}.
For general noise terms, it is difficult to determine whether a local pathwise solution to \eqref{SGCH problem} is globally defined. As shown in \cite{GlattHoltz-Vicol-2014-AP,Kim-2010-JFA,Kroker-Rohde-2012-ANM,Rockner-Zhu-Zhu-2014-SPTA,Tang-2018-SIMA}, the linear 
noise $h(t,u) {\rm d} W=\beta u {\rm d}W$, with $\beta\in\R\setminus\{0\}$ and $W$  being a standard 1-D Brownian motion, acts   dissipatively for many SPDEs. Motivated by these works,  we prove some results in Theorems \ref{Decay result} and \ref{Global existence result}   for the global dynamics of \eqref{SGCH problem} with   linear multiplicative 
noise
, that is
\begin{equation} \label{GCH linear noise}
\left\{\begin{aligned}
&{\rm d}u+\left[u^k\partial_xu+F(u)\right]{\rm d}t=b(t) u{\rm d}W,\quad x\in\T, \ t\in\R^+,\  k\geq 1,\\
&u(\omega,0,x)=u_0(\omega,x),\quad x\in\T.
\end{aligned} \right.
\end{equation}

%
%
%

\section{Preliminaries and Main Results}\label{Preliminaries and Main Results}
\subsection{Noise with $H^s$ coefficient}
We begin by introducing some notations. Throughout the paper, $(\Omega, \mathcal{F},\p)$ denotes a complete probability space, where $\p$ is the probability measure on $\Omega$ and $\mathcal{F}$ is a $\sigma$-algebra.
Let $t>0$. 
$\sigma\left\{\left(x(\tau),y(\tau)\right)_{\tau\in[0,t]}\right\}$ stands for the completion of the union $\sigma$-algebra generated by $\left(x(\tau),y(\tau)\right)$.
All stochastic integrals are defined in the It\^{o} sense and $\E \cdot$ is the mathematical expectation
of $\cdot$ with respect to $\p$. For some separable Banach space $X$, $\B(X)$ denotes the Borel sets of $X$,  and $\pr(X)$ stands for the collection of Borel probability measures on $X$. For $E\subseteq X$, $\textbf{1}_{E}$ is the indicator function on $E$.

$L^2(\T)$ is the usual square integrable function space on $\T$. For $s\in\R$,  $D^s=(1-\partial_{xx}^2)^{s/2}$ is defined by
$\widehat{D^sf}(k)=(1+k^2)^{s/2}\widehat{f}(k)$, where $\widehat{g}$ denote the  Fourier transform  of $g$ on $\T$. The Sobolev space $H^s(\T)$ is defined as
\begin{align*}
H^s(\T)\triangleq\{f\in L^2(\T):\|f\|_{H^s(\T)}^2=\sum_{k\in{\Z}}(1+k^2)^s|\widehat{f}(k)|^2<+\infty\},
\end{align*}
and the inner product    on $H^s(\T)$  
is $(f,g)_{H^s}\triangleq\sum_{k\in{\Z}}(1+k^2)^s\widehat{f}(k)\cdot\overline{\widehat{g}}(k)=(D^sf,D^sg)_{L^2}.$
For function spaces on $\T$, we will drop $\T$ if there is no ambiguity. We will use $\lesssim $ to denote estimates that hold up to some universal \textit{deterministic} constant which may change from line to line but whose meaning is clear from the context. For linear operators $A$ and $B$,  $[A,B]$ stands for the commutator of $A$ and $B$, i.e., $[A,B]=AB-BA$.

We briefly recall some aspects of the stochastic analysis theory  which we use below. We refer the readers to \cite{Prato-Zabczyk-2014-Cambridge,Gawarecki-Mandrekar-2010-Springer,Kallianpur-Xiong-1995-book} for an extended treatment of this subject.

We call
$\s=(\Omega, \mathcal{F},\p,\{\mathcal{F}_t\}_{t\geq0}, \W)$ a stochastic basis, where $\{\mathcal{F}_t\}_{t\geq0}$ is a right-continuous filtration on $(\Omega, \mathcal{F})$ such that $\{\mathcal{F}_0\}$
contains all the $\p$-negligible subsets and $\W(t)=\W(\omega,t),\omega\in\Omega$ is a cylindrical Brownian motion, defined on an auxiliary Hilbert space $U$, which is adapted to $\{\mathcal{F}_t\}_{t\geq0}$. Formally, if $\{e_k\}$ is a complete orthonormal basis of $U$ and $\{W_k\}_{k\geq1}$ is a sequence of mutually independent standard one-dimensional Brownian motions, then one may define 
\begin{equation*}
\W=\sum_{k=1}^\infty e_kW_k,\ \ \p-a.s.
\end{equation*}
To guarantee the convergence of the (formal) summation above, we consider a larger separable Hilbert space $U_0$ such that the canonical embedding $U\hookrightarrow U_0$ is Hilbert--Schmidt.
Then we have that for any $T>0$, cf. \cite{Prato-Zabczyk-2014-Cambridge,Gawarecki-Mandrekar-2010-Springer,Karczewska-1998-AUMCSS},
\begin{equation*}
\W=\sum_{k=1}^\infty e_kW_k\in C([0,T];U_0),\ \ \p-a.s.
\end{equation*}
To define the It\^{o} stochastic integral 
\begin{equation}\label{stochastic integral}
\int_0^\tau G{\rm d}\W=\sum_{k=1}^\infty\int_0^\tau G e_k{\rm d}W_k
\end{equation}
on $H^s$, it is required (see e.g.~\cite{Prato-Zabczyk-2014-Cambridge,Prevot-Rockner-2007-book})  for the predictable stochastic process $G$ to take values in the space of Hilbert-Schmidt operators from $U$ to $H^s$, denoted by $L_2(U; H^s)$.  For such $G$, \eqref{stochastic integral} is a well-defined continuous $H^s$-valued square integrable martingale such that
for all almost surely bounded stopping times $\tau$ and for all $v\in H^s$,
$$\left(\int_0^\tau G\ {\rm d}\W,v\right)_{H^s}=\sum_{k=1}^\infty\int_0^\tau (G e_k,v)_{H^s}\ {\rm d}W_k.$$
Moreover, the desirable Burkholder-Davis-Gundy inequality in our case turns out to be
\begin{align*} 
\E\left(\sup_{t\in[0,T]}\left\|\int_0^t G\ {\rm d}\W\right\|_{H^s}^p\right)
\leq C(p,s) \E\left(\int_0^T \|G\|^2_{L_2(U; H^s)}\ {\rm d}t\right)^\frac{p}{2},\ \ p\geq1.
\end{align*}
or in terms of the coefficients,
\begin{align*} 
\E\left(\sup_{t\in[0,T]}\left\|\sum_{k=1}^\infty\int_0^t Ge_k\ {\rm d}W_k\right\|_{H^s}^p\right)
\leq C(p,s) \E\left(\int_0^T \sum_{k=1}^\infty\|Ge_k\|^2_{H^s}\ {\rm d}t\right)^\frac{p}{2},\ \ p\geq1.
\end{align*}
Here we remark that the stochastic integral \eqref{stochastic integral} does not depend on the choice of the space $U_0$, cf.~\cite{Prato-Zabczyk-2014-Cambridge,Prevot-Rockner-2007-book}. For example, $U_0$ can be defined as 
\begin{equation*}
U_0=\left\{v=\sum_{k=1}^\infty a_ke_k:\sum_{k=1}^\infty\frac{a_k^2}{k^2}<\infty\right\},\ \ \|v\|_{U_0}=\sum_{k=1}^\infty\frac{a_k^2}{k^2}.
\end{equation*}
%
%
%
\subsection{Definitions of the solutions and stability of the exiting times}
We now  make the precise notion of a  pathwise solution to \eqref{SGCH problem}.
\begin{Definition}[Pathwise solutions]\label{pathwise solution definition}
	Let $\s=(\Omega, \mathcal{F},\p,\{\mathcal{F}_t\}_{t\geq0}, \W)$ be a fixed stochastic basis. Let $s>3/2$ and $u_0$ be an $H^s$-valued $\mathcal{F}_0$-measurable random variable $($relative to $\s)$.
	\begin{enumerate}
		\item A local pathwise solution to \eqref{SGCH problem} is a pair $(u,\tau)$, where $\tau\ge 0$ is a stopping time satisfying $\p\{\tau>0\}=1$ and
		$u:\Omega\times[0,\infty]\rightarrow H^s$  is an $\mathcal{F}_t$-predictable $H^s$-valued process satisfying
		\begin{equation*}
		u(\cdot\wedge \tau)\in C([0,\infty);H^s)\ \ \p-a.s.,
		\end{equation*}
		and for all $t>0$,
		\begin{equation*} 
		u(t\wedge \tau)-u(0)+\int_0^{t\wedge \tau}
		\left[u^k\partial_xu+F(u)\right]{\rm d}t'
		=\int_0^{t\wedge \tau}h(t',u){\rm d}\W\ \ \p-a.s.
		\end{equation*}
		\item The local pathwise solutions are said to be pathwise unique, if given any two pairs of local pathwise solutions $(u_1,\tau_1)$ and $(u_2,\tau_2)$ with $\p\left\{u_1(0)=u_2(0)\right\}=1,$ we have
		\begin{equation*}
		\p\left\{u_1(t,x)=u_2(t,x),\ \forall\ (t,x)\in[0,\tau_1\wedge\tau_2]\times \T\right\}=1.
		\end{equation*}
		\item Additionally, $(u,\tau^*)$ is called a maximal pathwise solution to \eqref{SGCH problem} if $\tau^*>0$ almost surely and if there is an increasing sequence $\tau_n\rightarrow\tau^*$ such that for any $n\in\N$, $(u,\tau_n)$ is a pathwise solution to \eqref{SGCH problem} and on the set $\{\tau^*<\infty\}$,
		\begin{equation*} 
		\sup_{t\in[0,\tau_n]}\|u\|_{H^s}\geq n.
		\end{equation*}
		\item If $(u,\tau^*)$ is a maximal pathwise solution and
		$\tau^*=\infty$ almost surely, then we say that the pathwise solution exists globally.
	\end{enumerate}
\end{Definition}
A major result of this paper is a  (negative) dependence statement for the initial data of the solutions. Precisely, it refers to the stability  of a point in time when the solution leaves a certain range. 
This point in time is called  exiting time and we introduce 
\begin{Definition}[Stability of exiting time]
	\label{Stability on exiting time}
	Let $s>3/2$ and $\s=(\Omega, \mathcal{F},\p,\{\mathcal{F}_t\}_{t\geq0}, \W)$ be a fixed stochastic basis. Let $u_0$ be an $H^s$-valued $\mathcal{F}_0$-measurable random variable such that $\E\|u_0\|^2_{H^s}<\infty$. Assume that $\{u_{0,n}\}$ is any sequence of $H^s$-valued $\mathcal{F}_0$-measurable random variables satisfying $\E\|u_{0,n}\|^2_{H^s}<\infty$. For each $n$, let $u$ and $u_n$ be the unique pathwise solutions to \eqref{SGCH problem} with initial value $u_0$ and $u_{0,n}$, respectively.\\
	 For any $R>0$ and $n\in\N$, we define the $R$-exiting times 
	\begin{align*}
	\tau^R_{n}:=\inf\left\{t\geq0: \|u_n(t)\|_{H^s}>R\right\} \mbox{ and }
	\tau^R:=\inf\left\{t\geq0: \|u(t)\|_{H^s}>R\right\}, 
	\end{align*}
	where $\inf \varnothing= \infty$. Furthermore,
	\begin{enumerate}
		\item If $u_{0,n}\rightarrow u_0$ in $H^{s}$ $\p-a.s.$ implies that
		\begin{align}
		\lim_{n\rightarrow\infty}\tau^R_{n}=\tau^R\ \  \p-a.s.,\label{hitting time convergence}
		\end{align}
		then the $R$-exiting time of $u$ is said to be stable.
		\item If $u_{0,n}\rightarrow u_0$ in $H^{s'}$ for all $s'<s$ almost surely  implies that \eqref{hitting time convergence} holds true,
		 the $R$-exiting time of $u$ is said to be strongly stable.
	\end{enumerate}

\end{Definition}
%
\subsection{Assumptions} 
For our results on existence of pathwise solutions,  on the stability of exiting times and global well-posedness of
 \eqref{gCH random dissipation}, 
we rely   in the following sections on  generic but slightly different assumptions concerning 
 the data in \eqref{gCH random dissipation}. These are summarized here, for a comment on possible relaxed versions we refer to Remark 
\ref{Remark for assumption}.

\begin{Assumption}\label{Assumption-1} 
	We assume that $h:[0,\infty)\times H^s\ni (t,u)\mapsto h(t,u)\in L_2(U; H^s)$ for $u\in H^s$ with $s\geq0$ such that if $u:\Omega\times [0,T]\rightarrow H^s$ is predictable, then $h(t,u)$ is also predictable. Furthermore, we assume the following:
		
		\begin{enumerate}

	\item There is an increasing and locally bounded function $f(\cdot):[0,+\infty)\rightarrow[0,+\infty)$ with $f(0)=0$ such that for any $t>0$ and $u\in H^s$ with $s>\frac{1}{2}$,
		\begin{align}\label{assumption 1 for h}
		\|h(t,u)\|_{L_2(U; H^s)}\leq f(\|u\|_{W^{1,\infty}})(1+\|u\|_{H^s}).
		\end{align}
		Particularly, if $h$ does not depend on $u$, i.e., the additive noise case, then the condition $f(0)=0$ can be removed.
		\item
		There is an increasing  locally bounded function $g(\cdot):[0,+\infty)\rightarrow[0,+\infty)$, such that for any $t>0$ and $u\in H^s$ with $s>\frac{1}{2}$,
		\begin{equation}\label{assumption 2 for h}
		\|h(t,u)-h(t,v)\|_{L_2(U; H^s)}\leq g(\|u\|_{H^{s}}+\|v\|_{H^{s}})\|u-v\|_{H^s}.
		\end{equation}
	\end{enumerate}
\end{Assumption}

\begin{Assumption}\label{Assumption-2} 
	When we consider the initial-data dependence problem for \eqref{SGCH problem} in Section \ref{Weak non-stability}, we need a similar assumption on $h(t,\cdot)$. For $s\geq 0$ and $u\in H^s$, $h:[0,\infty)\times H^s\ni (t,u)\mapsto h(t,u)\in L_2(U; H^s)$ for $u\in H^s$ with $s\geq0$ such that if $u:\Omega\times [0,T]\rightarrow H^s$ is predictable, then $h(t,u)$ is also predictable. Moreover, for all $t\geq0$, we assume
	\begin{equation}\label{assumption G}
	\|h(t,u)\|_{L_2(U; H^s)}\leq \|F(u)\|_{H^s},\ \
	\|h(t,u)-h(t,v)\|_{L_2(U; H^s)}\leq \|F(u)-F(v)\|_{H^s},
	\end{equation}
	where $F(\cdot)$ is defined by \eqref{F decomposition}.
\end{Assumption}

\begin{Assumption}\label{Assumption-3}
	When considering \eqref{GCH linear noise} with non-autonomous linear noise $b(t) u {\rm d}W$, we  assume that $b(t)\in C([0,\infty))$ such that $\displaystyle0<b_*\leq b^2(0)\leq\sup_{t\ge0}b^2(t)\leq b^*$ for some $b_*,b^*\in\R$.
\end{Assumption}

\subsection{Main results} In this section we summarize our major contributions providing  proofs later in the remainder of the paper.

\begin{Theorem}\label{Local pathwise solution}
	Let $s>3/2$, $k\geq1$ and let $h(t,u)$ satisfy Assumption \ref{Assumption-1}. For a given stochastic basis $\s=(\Omega, \mathcal{F},\p,\{\mathcal{F}_t\}_{t\geq0}, \W)$, if $u_0$ is an $H^s$-valued $\mathcal{F}_0$-measurable random variable satisfying $\E\|u_0\|^2_{H^s}<\infty$, then there is a local unique pathwise solution $(u,\tau)$ to \eqref{SGCH problem} in the sense of Definition \ref{pathwise solution definition} with
	\begin{equation}\label{L2 moment bound}
	u(\cdot\wedge \tau)\in L^2\left(\Omega; C\left([0,\infty);H^s\right)\right).
	\end{equation}
	Moreover, $(u,\tau)$ can be extended to a unique maximal pathwise solution $(u,\tau^*)$ with
	\begin{equation}\label{Blow-up criteria common}
	\textbf{1}_{\{\tau^*<\infty\}}=\textbf{1}_{\left\{\limsup_{t\rightarrow \tau^*}\|u(t)\|_{W^{1,\infty}}=\infty\right\}}\ \p-a.s.
	\end{equation}
\end{Theorem}

\begin{Remark} We remark here that $F_3(u)$ in \eqref{F decomposition} will disappear when $k=1$. The proof for Theorem \ref{Local pathwise solution} combines the techniques as employed in the papers \cite{Breit-Feireisl-Hofmanova-2018-CPDE,Breit-Feireisl-Hofmanova-2018-Book,Breit-Hofmanova-2016-IUMJ,Crisan-Flandoli-Holm-2018-JNS,Debussche-Glatt-Temam-2011-PhyD,GlattHoltz-Vicol-2014-AP,Tang-2018-SIMA}. However, the Faedo-Galerkin method used e.g.~in \cite{GlattHoltz-Vicol-2014-AP,Debussche-Glatt-Temam-2011-PhyD} cannot be utilized  directly since in \eqref{SGCH problem}, we do \textit{not} have additional constraints like {incompressibility}, which guarantees the global existence of an approximate solution (see, e.g. \cite{Flandoli-2008-SPDE-book,GlattHoltz-Vicol-2014-AP}).  Without this, we need to find a  positive lower bound for the lifespan of the approximate solutions, which is generally not clear. Particularly, for our case, this difficulty will be overcome by constructing a suitable approximation scheme and establishing an appropriate blow-up criterium, which applies not only for $u$, but also for the approximate solution $u_\e$. This idea is 
	 transferred from the recent work \cite{Crisan-Flandoli-Holm-2018-JNS} on deriving blow-up criteria. 
\end{Remark}

The next result addresses the dependence of the solution on initial data giving at least   a partial answer.

\begin{Theorem}[Weak instability]\label{weak instability}
	Consider the periodic initial value problem \eqref{SGCH problem}, where $F(\cdot)$ is given by \eqref{F decomposition} with $k\geq 1$. Let  $\s=(\Omega, \mathcal{F},\p,\{\mathcal{F}_t\}_{t\geq0}, \W)$ be a fixed stochastic basis and let $s>3/2$.  If $h$ satisfies Assumption \ref{Assumption-2}, then at least one of the following properties holds true.
	\begin{enumerate}
		\item For any $R\gg1$, the $R$-exiting time is \textbf{not} strongly stable for the zero solution in the sense of Definition \ref{Stability on exiting time};
		\item The solution map $u_0\mapsto u$ defined by solving \eqref{SGCH problem} is \textbf{not} uniformly continuous as a map from $L^2(\Omega,H^s)$ into $L^2\left(\Omega; C\left([0,T];H^s\right)\right)$ for any $T>0$. 
		More precisely, there exist two sequences of solutions $u_{1,n}(t)$ and $u_{2,n}(t)$, and two sequences of stopping times $\tau_{1,n}$ and $\tau_{2,n}$, such that
		
		\begin{itemize}
			
			\item $\p\{\tau_{i,n}>0\}=1$ for each $n>1$ and $i=1,2$. Besides,
			\begin{equation}
			\lim_{n\rightarrow\infty} \tau_{1,n}=\lim_{n\rightarrow\infty} \tau_{2,n}=\infty \  \ \p-a.s.\label{Non--uniform dependence condition 1}
			\end{equation}
			
			\item For $i=1,2$, $u_{i,n}\in C([0,\tau_{i,n}];H^s) $ $\p-a.s.$, and
			\begin{equation}\label{Non--uniform dependence condition 2}
			\E\left(\sup_{t\in[0,\tau_{1,n}]}\|u_{1,n}(t)\|^2_{H^{s}}+
			\sup_{t\in[0,\tau_{2,n}]}\|u_{2,n}(t)\|^2_{H^{s}}\right)\lesssim 1.
			\end{equation}
			
			\item At $t=0$,
			\begin{align}
			\lim_{n\rightarrow\infty}\E\|u_{1,n}(0)-u_{2,n}(0)\|^2_{H^{s}}=0.\label{Non--uniform dependence condition 3}
			\end{align}
			\item For any $T>0$,
			\begin{align}
			\liminf_{n\rightarrow\infty}\E
			\sup_{t\in[0,T\wedge\tau_{1,n}\wedge\tau_{2,n}]}\|u_{1,n}(t)-u_{2,n}(t)\|_{H^{s}}^2
			\gtrsim 
			\begin{cases}
			\displaystyle\left(\sup_{t\in[0,T]}|\sin t|\right)^2, 
			& {\rm if}  \ k\ \text{is odd},\\
			\displaystyle\left(\sup_{t\in[0,T]}\left|\sin \frac{t}{2}\right|\right)^2, 
			&  {\rm if} \ k\ \text{is even}.
			\end{cases}
			\label{Non--uniform dependence condition 4}
			\end{align}
		\end{itemize}
		
	\end{enumerate}
	
\end{Theorem}

\begin{Remark}
	To prove Theorem \ref{weak instability}, we assume that for some $R_0\gg1$, the $R_0$-exiting time is strongly stable at the zero solution.  Then we will construct an example to show that the solution map $u_0\mapsto u$ defined by \eqref{SGCH problem} is not uniformly continuous. This example involves the construction (for each  $s>3/2$) of two sequences of solutions which are converging at time zero but remain far apart at any later time. Therefore we will first construct two sequences of  approximation solutions $u^{i,n}$ ($i\in\{1,2\}$) such that the actual solutions $u_{i,n}$ starting from $u_{i,n}(0)=u^{i,n}(0)$ satisfy that as $n\rightarrow\infty$,
	\begin{align}
	\lim_{n\rightarrow\infty}\E\sup_{[0,\tau_{i,n}]}\|u_{i,n}-u^{i,n}\|_{H^s}
	=0,\label{non uniform remark equ}
	\end{align}
	where $u_{i,n}$ exists at least on $[0,\tau_{i,n}]$. If we obtain \eqref{non uniform remark equ}, then we can estimate the approximation solutions instead of the actual solutions and obtain  \eqref{Non--uniform dependence condition 4}. In deterministic case, for other works using the method of approximate solutions for studying dependence on initial data,  we refer the reader to \cite{Koch-Tzvetkov-2005-IMRN,Bona-Tzvetkov-2009-DCDS,Himonas-Kenig-Misiolek-2010-CPDE,Tang-Zhao-Liu-2014-AA,Tang-Liu-2015-ZAMP} and the references therein. However, in contrast to deterministic cases where lifespan estimate can be achieved (see (4.7)--(4.8) in \cite{Tang-Shi-Liu-2015-MM} and (3.8)--(3.9) in \cite{Tang-Zhao-Liu-2014-AA} for example), it is not clear $\inf_{n}\tau_{i,n}>0$ almost surely in stochastic setting. Therefore we are motivated to introduce the definition on the stability of the exiting time (see Definition \ref{Stability on exiting time}). 	
Then we find that  the property $\inf_{n}\tau_{i,n}>0$ can be connected with the stability property of the exiting time of the zero solution. Then we estimate the error in $H^{2s-\sigma}$ and $H^{\sigma}$ with suitable $\sigma$.  Different from \cite{Tang-2020-Arxiv}, the problem \eqref{SGCH problem} has nonlinearities of order $k+1$ and \eqref{non uniform remark equ} should depend on $k$.  Therefore more technical estimates are needed (see Lemma \ref{error estimate lemma} and \eqref{difference r's} for example).
	Finally we  use interpolation to derive \eqref{non uniform remark equ}. 
	Theorem \ref{weak instability} shows that one cannot expect too much for the issue of the dependence on initial data. More precisely, we cannot expect to improve the stability of the exiting time for the zero solution, and simultaneously to improve the continuous dependence of solutions on initial data. 
\end{Remark}

Finally, we focus on \eqref{GCH linear noise}.  For the issue of global existence, we have

\begin{Theorem}[Global existence: Case 1]\label{Decay result} Let $k\geq 1$ and $s>3/2$. Let $b(t)$ satisfy Assumption \ref{Assumption-3}  and $\s=(\Omega, \mathcal{F},\p,\{\mathcal{F}_t\}_{t\geq0}, W)$ be a fixed stochastic basis. Assume $u_0$ to be  an $H^s$-valued $\mathcal{F}_0$-measurable random variable satisfying $\E\|u_0\|^2_{H^s}<\infty$. Let $K$ be a constant such that $\|\cdot\|_{W^{1,\infty}}<K\|\cdot\|_{H^s}$. Then there is a $C=C(s)>1$ such that for any $R>1$ and $\lambda_1>2$, if $\|u_0\|_{H^s}<\frac{1}{RK}\left(\frac{b_*}{C\lambda_1}\right)^{1/k}$ almost surely,  then  \eqref{GCH linear noise} has a maximal pathwise solution $(u,\tau^*)$ satisfying  for any $\lambda_2>\frac{2\lambda_1}{\lambda_1-2}$ the estimate
	\begin{equation*}
	\p \left\{
	\|u(t)\|_{H^s}<\frac{1}{K}\left(\frac{b_*}{C\lambda_1}\right)^{1/k} 
	{\rm e}^{-\frac{\left((\lambda_1-2)\lambda_2-2\lambda_1\right)}{2\lambda_1\lambda2}\int_0^tb^2(t') {\rm d}t'}
	\ {\rm\ for\ all}\ t>0
	\right\}
	\geq 1-\left(\frac{1}{R}\right)^{2/\lambda_2}.
	\end{equation*}
	
\end{Theorem}

\begin{Remark} Here we notice that if $k=1$, then $F_3(u)$ in \eqref{F decomposition} will disappear.
	Motivated by the recent papers \cite{GlattHoltz-Vicol-2014-AP,Rockner-Zhu-Zhu-2014-SPTA,Tang-2018-SIMA}, where the linear noise $\beta u {\rm d}W$ with $\beta\in\R\setminus\{0\}$ is considered,   we focus on the non-autonomous linear multiplicative noise case, namely \eqref{GCH linear noise}. We transform  \eqref{GCH linear noise} to a non-autonomous random system  \eqref{periodic Cauchy problem transform}. 
	Although the stochastic integral is absent in \eqref{periodic Cauchy problem transform}, to extend the deterministic results to the stochastic setting, we need to overcome a few technical difficulties since the system is not only random but also non-autonomous. In this work, we manage to gain some estimates and asymptotic limits of the Girsanov type processes (see e.g., \eqref{global time tau}, \eqref{tau 2 Girsanov}, \eqref{tau 1 > tau 2} and Lemma \ref{eta Lemma}), which enable us to apply the energy estimate pathwisely (namely for a.e. $\omega\in\Omega$) and obtain Theorem \ref{Decay result}. 
\end{Remark}

\begin{Theorem}[Global existence: Case 2]\label{Global existence result} Let $\s=(\Omega, \mathcal{F},\p,\{\mathcal{F}_t\}_{t\geq0}, W)$ be a fixed stochastic basis. 
	Let $s>3$ and $b(t)$ satisfy Assumption \ref{Assumption-3}. Assume $u_0$ to be  an $H^s$-valued, $\mathcal{F}_0$-measurable random variable satisfying $\E\|u_0\|^2_{H^s}<\infty$.	If  $u_0$ satisfies
	\begin{equation*} 
	\p\left\{(1-\partial_{xx}^2)u_0(x)>0,\ \forall x\in\T\right\}=p,\ \
	\p\left\{(1-\partial_{xx}^2)u_0(x)<0,\ \forall x\in\T\right\}=q,
	\end{equation*}
	for some $p,q\in[0,1]$, then the corresponding maximal pathwise solution $(u,\tau^*)$ to \eqref{GCH linear noise} satisfies
	\begin{equation*}
	\p\{\tau^*=\infty\}\geq p+q.
	\end{equation*}
	That is to say, $\p\left\{u\ {\rm exists\ globally}\right\}\geq p+q.$
\end{Theorem}


\section{Proof for Theorem \ref{Local pathwise solution}}

\subsection{Blow-up criteria}
Let us postpone the proof for existence and uniqueness of solutions to \eqref{SGCH problem} to Section \ref{Existence and uniqueness}. We will first prove  the blow-up criteria,  since some  estimates will be used later. Motivated by \cite{Crisan-Flandoli-Holm-2018-JNS}, we first consider the relationship between the blow-up time of $\|u(t)\|_{H^s}$ and the blow-up time of $\|u(t)\|_{W^{1,\infty}}$ for \eqref{SGCH problem}. Even though one might expect that the $\|u(t)\|_{H_s}$ norm may blow up earlier than $\|u(t)\|_{W^{1,\infty}}$, the following result shows that this is not true.

\begin{Lemma}\label{blow-up criteria lemma}
	Let $(u,\tau^*)$ be the unique maximal pathwise solution to \eqref{SGCH problem}. Then the real-valued stochastic process $\|u\|_{W^{1,\infty}}$ is also $\mathcal{F}_t$-adapted. Besides, for any $m,n\in\Z^{+}$, define
	\begin{align*}
	\tau_{1,m}=\inf\left\{t\geq0: \|u(t)\|_{H^s}\geq m\right\},\ \ \
	\tau_{2,n}=\inf\left\{t\geq0: \|u(t)\|_{W^{1,\infty}}\geq n\right\}.
	\end{align*}
	Moreover, let  $\displaystyle\tau_1=\lim_{m\rightarrow\infty}\tau_{1,m}$ and $\displaystyle\tau_2=\lim_{n\rightarrow\infty}\tau_{2,n}$,  then we have
	$$
	\tau_{1}=\tau_{2} \ \ \p-a.s.
	$$
	As a corollary, 
	$\textbf{1}_{\left\{\lim_{t\rightarrow \tau^*}\|u(t)\|_{W^{1,\infty}}=\infty\right\}}=\textbf{1}_{\{\tau^*<\infty\}}$ $\p-a.s.$
\end{Lemma}
\begin{proof}
	To begin with, we see that $u(\cdot\wedge \tau)\in C([0,\infty);H^s)$ means that for any $t\in[0,\tau]$,
	$$[u(t)]^{-1}(Y)=[u(t)]^{-1}(H^s\cap Y),\ \forall\ Y\in\B(W^{1,\infty}).$$ Therefore $u(t)$, as a $W^{1,\infty}$-valued process, is also $\mathcal{F}_t$-adapted. We then infer from the embedding $H^s\hookrightarrow W^{1,\infty}$ for $s>3/2$ that for some $M>0$ and $m\in\N$,
	\begin{align*}
	\sup_{t\in[0,\tau_{1,m}]}\|u(t)\|_{W^{1,\infty}}\leq M\sup_{t\in[0,\tau_{1,m}]}\|u(t)\|_{H^s}
	\leq ([M]+1)m,
	\end{align*}
	where $[M]$ means the integer part of $M$. Therefore we have
	$\tau_{1,m}\leq\tau_{2,([M]+1)m}\leq \tau_2$ $\p-a.s.,$
	which means that
	$\tau_{1}\leq \tau_2$ $\p-a.s.$
	Now we only need to prove $\tau_{2}\leq \tau_1$ $\p-a.s.$ It is easy to see that for all $n_1,n_2\in\Z^{+}$,
	\begin{align*}
	\left\{\sup_{t\in[0,\tau_{2,n_1}\wedge n_2]}\|u(t)\|_{H^s}<\infty\right\}
	=\bigcup_{m\in\Z^{+}}\left\{\sup_{t\in[0,\tau_{2,n_1}\wedge n_2]}\|u(t)\|_{H^s}<m\right\}
	\subset\bigcup_{m\in\Z^{+}}\left\{\tau_{2,n_1}\wedge n_2\leq\tau_{1,m}\right\}.
	\end{align*}
	Since $$\bigcup_{m\in\Z^{+}}\left\{\tau_{2,n_1}\wedge n_2\leq\tau_{1,m}\right\}\subset\left\{\tau_{2,n_1}\wedge n_2\leq\tau_{1}\right\},$$
	we see that if 
	\begin{align}
	\p\left\{\sup_{t\in[0,\tau_{2,n_1}\wedge n_2]}\|u(t)\|_{H^s}<\infty\right\}=1,\ \
	\forall\ n_1,n_2\in\Z^{+}\label{tau2<tau1 condition}
	\end{align}
	holds true, then for all $n_1,n_2\in\Z^{+}$, $\p\left\{\tau_{2,n_1}\wedge n_2\leq\tau_{1}\right\}=1$ and
	\begin{align}
	\p\left\{\tau_2\leq\tau_1\right\}
	=\p\left\{\bigcap_{n_1\in\Z^{+}}\left\{\tau_{2,n_1}\leq \tau_{1}\right\}\right\}
	=\p\left\{\bigcap_{n_1,n_2\in\Z^{+}}\left\{\tau_{2,n_1}\wedge n_2\leq \tau_{1}\right\}\right\}=1.\label{tau2<tau1}
	\end{align}
	Since \eqref{tau2<tau1} requires the assumption \eqref{tau2<tau1 condition}, it suffices to prove \eqref{tau2<tau1 condition}. However, we can not directly apply the It\^{o} formula for $\|u\|^2_{H^s}$ to get control of $\E\|u(t)\|_{H^s}^2$ since we only have $u\in H^s$ and $u^ku_x\in H^{s-1}$.
%
 Therefore the well known It\^{o} formula in general Hilbert space can not be used directly, see \cite[Theorem 4.32]{Prato-Zabczyk-2014-Cambridge} or \cite[Theorem 2.10]{Gawarecki-Mandrekar-2010-Springer}  for example. 
	We will use the mollifier operator $T_\e$ defined in the appendix to overcome this obstacle. Indeed, applying $T_\e$ to \eqref{SGCH problem} and using the It\^{o} formula for $\|T_\e u\|^2_{H^s}$, we have that for any $t>0$,	
	\begin{align}
	{\rm d}\|T_\e u(t)\|^2_{H^s}
	=& \left(T_\e h(t,u){\rm d}\W,T_\e u\right)_{H^s}
	-2\left(D^sT_\e\left[u^ku\right],D^sT_\e u\right)_{L^2}{\rm d}t\notag\\
	&-2\left(D^sT_\e F(u),D^sT_\e u\right)_{L^2}{\rm d}t
	+ \|T_\e h(t,u)\|_{L_2(U; H^s)}^2{\rm d}t.\notag
	\end{align}
	Therefore 
	for any $n_1,n_2\geq 1$ and $t\in[0,\tau_{2,n_1}\wedge n_2]$,
	\begin{align*}
	\|T_\e u(t)\|^2_{H^s}-\|T_\e u(0)\|^2_{H^s}
	=&2\sum_{j=1}^{\infty}\int_0^{t}
	\left(D^sT_\e h(t,u)e_j,D^sT_\e u\right)_{L^2}{\rm d}W_j\nonumber\\
	&-2\int_0^{t}
	\left(D^sT_\e 
	\left[u^k\partial_xu\right],D^sT_\e u\right)_{L^2}{\rm d}t\nonumber\\
	&-2\int_0^{t}
	\left(D^sT_\e F(u),D^sT_\e u\right)_{L^2}{\rm d}t\nonumber\\
	&+\int_0^{t}
	\sum_{k=1}^\infty\|D^sT_\e h(t,u)e_j\|_{L^2}^2{\rm d}t\nonumber\\
	=&\int_0^{t}\sum_{j=1}^{\infty} L_{1,j}{\rm d}W_j+\sum_{i=2}^4\int_0^tL_i{\rm d}t,
	\end{align*}
	where $\{e_k\}$ is the complete orthonormal basis of $U$.
	On account of the Burkholder-Davis-Gundy inequality, we arrive at 
	\begin{align*}
	\E\sup_{t\in[0,\tau_{2,n_1}\wedge n_2]}\|T_\e u(t)\|^2_{H^s}\leq\E\|T_\e u_0\|^2_{H^s}
	+C\E\left(\int_0^{\tau_{2,n_1}\wedge n_2}\left|\sum_{j=1}^{\infty} L_{1,j}\right|^2{\rm d}t\right)^{\frac12}
	+\sum_{i=2}^4\E\int_0^{\tau_{2,n_1}\wedge n_2}|L_i|{\rm d}t.
	\end{align*}
	Then \eqref{assumption 1 for h}, \eqref{mollifier property 5} and the stochastic Fubini theorem \cite{Prato-Zabczyk-2014-Cambridge} lead to
	\begin{align*}
	\E\left(\int_0^{\tau_{2,n_1}\wedge n_2}\left|\sum_{j=1}^{\infty} L_{1,j}\right|^2{\rm d}t\right)^{\frac12}
	\leq& \frac12\E\sup_{t\in[0,\tau_{2,n_1}\wedge n_2]}\|T_\e u\|_{H^s}^2
	+Cf^2(n_1)\int_0^{n_2}\left(1+\E\|u\|_{H^s}^2\right){\rm d}t.
	\end{align*}
	For $L_2$,  we notice that $T_\e$ satisfies \eqref{mollifier property 3}, \eqref{mollifier property 4} and \eqref{mollifier property 5}. Then it follows from Lemmas   \ref{Te commutator}, \ref{Kato-Ponce commutator estimate} and \ref{Moser estimate}, integration by parts and $H^s\hookrightarrow W^{1,\infty}$ that
	\begin{align*}
	&\left(D^sT_\e 
	\left[u^ku_x\right],D^sT_\e u\right)_{L^2}\notag\\
	=&\left(\left[D^s,
	u^k\right]u_x,D^sT^2_\e u\right)_{L^2}+
	\left([T_\e,u^k]D^su_x, D^sT_\e u\right)_{L^2}
	+\left(u^kD^sT_\e u_x, D^sT_\e u\right)_{L^2}\notag\\
	\leq&  C\|u\|^{k}_{W^{1,\infty}}\|u\|^2_{H^s},
	\end{align*}
	which implies
	\begin{align*}
	\E\int_0^{\tau_{2,n_1}\wedge n_2}|L_2|\ {\rm d}t\leq Cn^k_1\int_0^{n_2}\left(1+\E\|u\|_{H^s}^2\right){\rm d}t.
	\end{align*}
	Similarly, it follows from Lemma \ref{F lemma} and the assumption \eqref{assumption 1 for h} that
	\begin{align*}
	\E\int_0^{\tau_{2,n_1}\wedge n_2}|L_3|+|L_4|\ {\rm d}t\leq C(n_1^k+f^2(n_1))\int_0^{n_2}\left(1+\E\|u\|_{H^s}^2\right){\rm d}t.
	\end{align*}
	Therefore we combine the above estimates with using \eqref{mollifier property 5} to have
	\begin{align*}
	\E\sup_{t\in[0,\tau_{2,n_1}\wedge n_2]}\|T_\e u(t)\|^2_{H^s}
	\leq C\E\|u_0\|^2_{H^s}+ C\int_0^{n_2}\left(1+\E\sup_{t'\in[0,t\wedge\tau_{2,n_1}]}\|u(t')\|_{H^s}^2\right){\rm d}t,
	\end{align*}
	where $C=C(n_1,k)$ through $n_1^k$ and $n_1^k+f^2(n_1)$.   Since the right hand side of the above estimate does not depend on $\e$, and
	$T_\e u$ tends to $u$ in $C\left([0,T],H^{s}\right)$ for any $T>0$ almost surely as $\e\rightarrow0$, one can send $\e\rightarrow0$ to obtain
	\begin{align}\label{u Hs estimate}
	\E\sup_{t\in[0,\tau_{2,n_1}\wedge n_2]}\|u(t)\|^2_{H^s}
	\leq C\E\|u_0\|^2_{H^s}+ C\int_0^{n_2}\left(1+\E\sup_{t'\in[0,t\wedge\tau_{2,n_1}]}\|u(t')\|_{H^s}^2\right){\rm d}t.
	\end{align}
	Then  Gronwall's inequality shows that for each $n_1,n_2\in\Z^{+}$, there is a constant $C=C(n_1,n_2,u_0,k)>0$ such that $$\E\sup_{t\in[0,\tau_{2,n_1}\wedge n_2]}\|u(t)\|^2_{H^s}<C(n_1,n_2,u_0,k),$$ which gives \eqref{tau2<tau1 condition}.
\end{proof}
	
	\begin{proof}[Proof for \eqref{Blow-up criteria common}]
	Let us assume the existence and uniqueness first. We notice that for fixed $m,n>0$, even if $\p\{\tau_{1,m}=0\}$ or $\p\{\tau_{2,n}=0\}$ is larger than $0$, for a.e. $\omega\in\Omega$, there is $m>0$ or $n>0$ such that $\tau_{1,m},\tau_{2,n}>0$. By continuity of $\|u(t)\|_{H^s}$ and the uniqueness of $u$, it is easy to check that $\tau_1=\tau_2$ is actually the maximal existence time $\tau^*$ of $u$ in the sense of Definition \ref{pathwise solution definition}. Consequently, we obtain the desired blow-up criteria.
\end{proof}

\subsection{Sketch of the proof for Theorem \ref{Local pathwise solution}}\label{Existence and uniqueness}

Since the main idea of proving Theorem \ref{Local pathwise solution} follows standard ideas and it is very similar to \cite{Tang-2018-SIMA,Tang-2020-Arxiv}, here we only give a sketch of the main steps.

\begin{enumerate}

	\item (Approximate solutions) The first step is to construct a suitable approximation scheme. For any $R>1$, we let $\chi_R(x):[0,\infty)\rightarrow[0,1]$ be a $C^{\infty}$ function such that $\chi_R(x)=1$ for $x\in[0,R]$ and $\chi_R(x)=0$ for $x>2R$.
	Then we consider the following cut-off problem on $\T$,
	\begin{equation} \label{cut-off problem}
	\left\{\begin{aligned}
	&{\rm d}u+\chi_R(\|u\|_{W^{1,\infty}})\left[u^k\partial_xu+F(u)\right]{\rm d}t=\chi_R(\|u\|_{W^{1,\infty}})h(t,u){\rm d}\W,\\
	&u(\omega,0,x)=u_0(\omega,x)\in H^{s}.
	\end{aligned} \right.
	\end{equation}
	From Lemma \ref{F lemma}, we see that the nonlinear term $F(u)$ preserves the $H^s$-regularity of $u\in H^s$ for any $s>3/2$. However, to apply the theory of SDEs in Hilbert space to \eqref{cut-off problem}, we will have to mollify the transport term $u^k\partial_xu$ since the product $u^k\partial_xu$ loses one regularity. To this end, we consider the following approximation scheme:
	\begin{equation} \label{approximate problem}
	\left\{\begin{aligned}
	{\rm d}u+G_{1,\e}(u){\rm d}t&=G_{2}(u){\rm d}\W,\ \ x\in\T,\ t>0,\\
	G_{1,\e}(u)&=\chi_R(\|u\|_{W^{1,\infty}})\left[J_{\varepsilon}
	\left((J_{\varepsilon}u)^k\partial_xJ_{\varepsilon}u\right)+F(u)\right],\\
	G_{2}(u)&=\chi_R(\|u\|_{W^{1,\infty}})h(t,u),\\
	u(0,x)&=u_0(x)\in H^{s}(\T),
	\end{aligned} \right.
	\end{equation}
	where $J_{\e}$ is the Friedrichs mollifier defined in Appendix \ref{appendix}. After mollifying the transport term $u^ku_x$,  for a fixed stochastic basis $\s=(\Omega, \mathcal{F},\p,\{\mathcal{F}_t\}_{t\geq0}, W)$ and for $u_0\in L^2(\Omega;H^s)$ with $s>3$, according to the existence theory of SDE in Hilbert space  (see for example \cite[Theorem 4.2.4 with Example 4.1.3]{Prevot-Rockner-2007-book} and \cite{Kallianpur-Xiong-1995-book}), \eqref{approximate problem} admits a unique solution $u_{\e}\in C([0,T_\e),H^s)$ $\p-a.s.$  Moreover, the uniform $L^{\infty}(\Omega;W^{1,\infty})$ condition provided by the cut-off function $\chi_R$ enables us to split the expectation $\E(\|u_\e\|_{H^{s}}^2\|u_\e\|_{W^{1,\infty}})$ to close the  {\it a priori} $L^2(\Omega;H^s)$ estimate for $u_\e$.    Then we can go along the lines as we prove Lemma \ref{blow-up criteria lemma} to find that for each fixed $\e$, if $T_\e<\infty$, then
	$
	\limsup_{t\rightarrow T_\e}\|u_\e(t)\|_{W^{1,\infty}}=\infty$.
	Due to the cut-off in \eqref{approximate problem}, for a.e. $\omega\in\Omega$, $\|u_\e(t)\|_{W^{1,\infty}}$ is always bounded and hence $u_{\e}$ is actually a global in time solution, that is, $u_{\e}\in C([0,\infty),H^s)$ $\p-a.s.$ We remark here that the global existence of $u_\e$ is necessary in our framework due to the lack of lifespan estimate in the stochastic setting. Otherwise we will have to prove $\p\{\inf_{\e>0}T_\e>0\}=1$, which is  not clear in general.
	
	\item (Pathwise solution to the cut-off problem in $H^s$ with $s>3$) We pass to the limit $\e\rightarrow0$. By applying the stochastic compactness arguments from Prokhorov's  and Skorokhod's theorems, we  obtain the almost sure convergence for a new approximate solution $(\widetilde{u_{\e}},\widetilde{W_{\e}})$ defined on a \emph{new} probability space. By virtue of a refined martingale representation theorem \cite[Theorem A.1]{Hofmanova-2013-SPTA},
	we may send $\e\rightarrow0$ in $(\widetilde{u_{\e}},\widetilde{W_{\e}})$ to build a global martingale solution in $H^s$ with $s>3$ to the cut-off problem. Finally, since $F$ satisfies the estimates as in Lemma \ref{F lemma} and $h$ satisfies Assumption \ref{Assumption-1}, one can obtain the pathwise uniqueness easily.  Then the  Gy\"{o}ngy--Krylov characterization \cite{Gyongy-Krylov-1996-PTRF} of the convergence in probability can be applied here to prove the convergence of the original approximate solutions. For more details, we refer to \cite{Tang-2018-SIMA,Tang-2020-Arxiv}. 
%

	\item (Remove the cut-off and extend the range of $s$ to $s>3/2$) When $u_0\in L^{\infty}(\Omega,H^s)$ with $s>3/2$, by mollifying the initial data, we obtain a sequence of regular solutions $\{u_n\}_{n\in\N}$ to \eqref{SGCH problem}. Motivated by \cite{GlattHoltz-Ziane-2009-ADE,Tang-2018-SIMA}, one can prove that there is a subsequence (still denoted by $u_n$) such that for some almost surely positive stopping time $\tau$,
	\begin{equation*}
	\lim_{n\rightarrow\infty}\sup_{t\in[0,\tau]}\|u_{n}-u\|_{H^s}=0\ \ \p-a.s.,
	\end{equation*}
	and
	\begin{equation}\label{u bound by u0}
	\sup_{t\in[0,\tau]}\|u\|_{H^s}\leq \|u_0\|_{H^s}+2 \ \ \p-a.s.
	\end{equation}
	Then we can pass the limit $n\rightarrow\infty$ to prove that $(u,\tau)$ is a solution to \eqref{SGCH problem}. Besides, using a cutting argument, as in \cite{GlattHoltz-Vicol-2014-AP,GlattHoltz-Ziane-2009-ADE,Breit-Feireisl-Hofmanova-2018-CPDE}, enables us to remove the $L^\infty(\Omega)$ assumption on $u_0$. More precisely, when $\E\|u_0\|^2_{H^s}<\infty$, we consider the decomposition $$\Omega_m=\{m-1\leq\|u_0\|_{H^s}<m\},\ m\geq1.$$
	Since $\E\|u_0\|^2_{H^s}<\infty$,  $\bigcup_{m\geq1}\Omega_m$ is a set of full measure and $1=\sum_{m\geq1}\textbf{1}_{\Omega_m}$ $\p-a.s.$ Therefore we have
	$$u_0(\omega,x)=\sum_{m\geq1}u_{0,m}(\omega,x)
	=\sum_{m\geq1}u_0(\omega,x)\textbf{1}_{\Omega_m}\ \ \p-a.s.$$ For each initial value $u_{0,m}$, we let $(u_m,\tau_m)$ be the pathwise unique solution to \eqref{SGCH problem} satisfying \eqref{u bound by u0}.   Moreover, as $F(0)=0$ and $h(0)=0$ (cf. \eqref{assumption 1 for h}), direct computation shows that
	\begin{equation*}
	\left(u=\sum_{m\geq1}u_m\textbf{1}_{m-1\leq\|u_0\|_{H^s}<m},\ \
	\tau=\sum_{m\geq1}\tau_m\textbf{1}_{m-1\leq\|u_0\|_{H^s}<m}\right)
	\end{equation*}
	is the unique pathwise solution to \eqref{SGCH problem} corresponding to the initial condition $u_0$. Since $(u_m,\tau_m)$ satisfies \eqref{u bound by u0}, we have $u(\cdot\wedge \tau)\in C\left([0,\infty);H^s\right)$ $\p-a.s.$ and
	\begin{align*}
	\sup_{t\in[0,\tau]}\|u\|_{H^s}^2
	=&\sum _ { m= 1 }^{ \infty }\textbf{1}_{m-1\leq\|u_0\|_{H^s}<m} 
	\sup _ { t \in \left[ 0 , \tau _ { m } \right] } \left\| u _ { m } \right\|_{H^s} ^ { 2 }\\
	\leq &
		C \sum _ { m =1} ^ { \infty } \textbf{1}_{m-1\leq\|u_0\|_{H^s}<m}  
		\left( 4+ \left\| u_{0,m} \right\|_{H^s} ^ { 2 } \right)
	=C\left( 4+ \left\| u_{0} \right\|_{H^s} ^ { 2 } \right).
	\end{align*}
	Taking expectation in the above inequality, we obtain \eqref{L2 moment bound}. Since the passage from $(u,\tau)$ to a unique maximal pathwise solution $(u,\tau^*)$ in the sense of Definition \ref{pathwise solution definition} can be carried out as in  \cite{Crisan-Flandoli-Holm-2018-JNS,GlattHoltz-Vicol-2014-AP,GlattHoltz-Ziane-2009-ADE,Rockner-Zhu-Zhu-2014-SPTA}, we omit the details. The proof for Theorem \ref{Local pathwise solution} is finished.

\end{enumerate}

\begin{Remark}\label{Remark for assumption}
	Changing the  growth assumption \eqref{assumption 1 for h}  to  $\|h(t,u)\|_{L_2(U; H^s)}\leq f(\|u\|_{H^{s'}})(1+\|u\|_{H^s})$ with $H^s\hookrightarrow H^{s'}\hookrightarrow W^{1,\infty}$ leads to  a criticality statement.   One can then 
	go along the lines in Lemma \ref{blow-up criteria lemma} to find $\textbf{1}_{\{\tau^*<\infty\}}=\textbf{1}_{\left\{\limsup_{t\rightarrow \tau^*}\|u(t)\|_{H^{s'}}=\infty\right\}}$ $\p-a.s.$ 
	 Using another cut-off $\chi_{R}(\|u\|_{H^{s'}})$ in the approximate scheme \eqref{cut-off problem} and  in \eqref{approximate problem}, the other part of the proof  also allows us  to establish a  local existence and uniqueness result. 
	  The difference condition \eqref{assumption 2 for h}  is essential to guarantee pathwise uniqueness.
\end{Remark}
%
%
%

\section{Proof for Theorem \ref{weak instability}}\label{Weak non-stability}

Now we are going to prove Theorem \ref{weak instability}. To this end, it suffices to show that if for some $R_0\gg1$, the $R_0$-exiting time is strongly stable at the zero solution, then the solution map $u_0\mapsto u$  is not uniformly continuous. 
We restrict our attention to $k\geq2$ since the case $k=1$, i.e., the stochastic CH equation, has been obtained in \cite{Tang-2020-Arxiv}. 

\subsection{Approximate solutions and associated estimates}
Define the approximate solutions as
\begin{equation*}
u^{l, n}=l n^{-\frac{1}{k}}+n^{-s}\cos \theta \ \text{and}\ \theta=nx-l t, \ n\in\Z^{+},
\end{equation*}
where
\begin{equation}\label{l n condition}
l\in\{-1,1\}\ \text{if}\ k\ \text{is odd};\ \  l\in\{0,1\} \ \text{if}\ k\ \text{is even}.
\end{equation}
Substituting $u^{l,n}$ into \eqref{SGCH problem}, we see that the error $E^{l, n}(t)$ can be defined as
\begin{align}\label{error definition}
E^{l, n}(t)=u^{l, n}(t)-u^{l, n}(0)+\int_{0}^{t}\left[\left(u^{l, n}\right)^k\partial_x u^{l, n}+F(u^{l, n})\right]{\rm d}t'-\int_{0}^{t}h(t',u^{l, n}){\rm d}\W.
\end{align}
Now we analyze the error as follows.

\begin{Lemma}\label{error estimate lemma}
	Let $s>3/2$. For $n\gg 1$, $\delta\in(1/2,\min\{s-1,3/2\})$ and any $T>0$, there is a $C=C(T)>0$ such that
	\begin{align*}
	\E\sup_{t\in[0,T]}\|E^{l, n}(t)\|^2_{H^\delta}\leq
	Cn^{-2r_s}.
	\end{align*}
	Here $r_{s}$ is a parameter with
	\begin{equation}\label{rs>0}
	0<r_s=
	\begin{cases}
	\displaystyle
	2s-\delta-\frac{k+1}{k}, & {\rm if}  \ \displaystyle\frac{3}{2}<s\leqslant \frac{2k+1}{k},\\
	s-\delta+1, &  {\rm if} \  \displaystyle s>\frac{2k+1}{k}.
	\end{cases}
	\end{equation}
\end{Lemma}
\begin{proof}
	Since \eqref{l n condition} implies $l^k=l$. This means
	\begin{align*}
	&u^{l, n}(t)-u^{l, n}(0)+\int_{0}^{t}\left(u^{l, n}\right)^k\partial_x u^{l, n}\ {\rm d}t'\\
	=&u^{l, n}(t)-u^{l, n}(0)+\int_{0}^t\left(\sum_{j=0}^{k}\left(\begin{array}{l}{k} \\ {j}\end{array}\right) \left( l n^{-\frac{1}{k}}\right)^{k-j} n^{-sj}\cos^{j}\theta\right)\left(-n^{-s+1}\sin\theta\right)\ {\rm d}t'\\
	=&\int_{0}^t\left(\sum_{j=1}^{k}\left(\begin{array}{l}{k} \\ {j}\end{array}\right) \left( l n^{-\frac{1}{k}}\right)^{k-j} n^{-sj}\cos^{j}\theta\right)\left(-n^{-s+1}\sin\theta\right)\ {\rm d}t'
	\triangleq\int_{0}^tT_{n,k}\ {\rm d}t'.
	\end{align*}
	Then we have
	\begin{equation}
	E^{l, n}(t)-
	\int_{0}^{t}\left[T_{n,k}+F(u^{l, n})\right]{\rm d}t'
	+\int_{0}^{t}h(t',u^{l, n}){\rm d}\W=0.\label{Error equation}
	\end{equation}
We first notice that
	\begin{align}\label{error 1}
	\left\|T_{n,k}\right\|_{H^{\delta}}
	\lesssim & \sum_{j=1}^{k}n^{\frac{j}{k}-sj-s}\left\|\cos^j\theta\sin\theta\right\|_{H^{\delta}}\nonumber\\
	\lesssim & \sum_{j=1}^{k}n^{\frac{j}{k}-sj-s-1}\left\| \cos^{j+1}\theta \right\|_{H^{\delta+1}}\nonumber\\
	\lesssim &\max_{1\leq j\leq n}\{n^{\frac{j}{k}-sj-s+\delta}\}
	\lesssim  n^{\frac{1}{k}-2s+\delta}\lesssim  n^{1-2s+\delta}.
	\end{align}
	Recall that $F(\cdot)$ is given by \eqref{F decomposition}. Since $(1-\partial^2_{xx})^{-1}$ is bounded from $H^\delta$ to $H^{\delta+2}$, we can use Lemmas \ref{cos sin approximate estimate}  and \ref{Moser estimate} to   estimate $\|F_i(u^{l, n})\|_{H^{\delta}}$ ($i=1,2,3$) as follows.
	\begin{align}
	\|F_1(u^{l, n})\|_{H^{\delta}}\lesssim & \sum_{j=0}^{k}n^{\frac{j}{k}-sj-s}\left\|\cos^j\theta \sin\theta\right\|_{H^{\delta-2}}\notag\\
	\lesssim & \sum_{j=0}^{k}n^{\frac{j}{k}-sj-s-1}\left\|\cos^{j+1}\theta\right\|_{H^{\delta-1}}\nonumber\\
	\lesssim & \sum_{j=0}^{k}n^{\frac{j}{k}-sj-s-1}\left\|\cos^{j+1}\theta\right\|_{H^{\delta}}\nonumber\\
	\lesssim & \max_{0\leq j\leq k}\{n^{\frac{j}{k}-sj-s-1+\delta}\}
	\lesssim  n^{-s+\delta-1}.\label{F1 error 1}
	\end{align}
	\begin{align}\label{F2 error 1}
	\|F_2(u^{l, n})\|_{H^{\delta}}\lesssim & \sum_{j=0}^{k-1}n^{\frac{j+1}{k}-sj-2s+1}
	\left\|\cos^j\theta\sin^2\theta \right\|_{H^{\delta-1}}\nonumber\\
\lesssim & \sum_{j=0}^{k-1}n^{\frac{j+1}{k}-sj-2s+1}
\left\|\cos^j\theta\sin^2\theta \right\|_{H^{\delta}}\nonumber\\
	\lesssim & \max_{0\leq j\leq k-1}\{n^{\frac{j+1}{k}-sj-2s+\delta+1}\}
	\lesssim n^{-2s+\delta+1+\frac{1}{k}}.
	\end{align}
	\begin{align}\label{F3 error 1}
	\|F_3(u^{l, n})\|_{H^{\delta}}
	\lesssim & \sum_{j=0}^{k-2}n^{\frac{j+2}{k}-sj-3s+2}\left\|\cos^j\theta \sin^3\theta \right\|_{H^{\delta-2}}\nonumber\\
	\lesssim & \sum_{j=0}^{k-2}n^{\frac{j+2}{k}-sj-3s+2}\left\|\cos^j\theta \sin^3\theta \right\|_{H^{\delta}}\nonumber\\
	\lesssim & \max_{0\leq j\leq k-2}\{n^{\frac{j+2}{k}-sj-3s+\delta+2}\}
	\lesssim n^{-3s+\delta+\frac{2}{k}+2}.
	\end{align}
	In the above estimates, we used the fact that $F_3(\cdot)$ appears only for $k\geq 2$. Combining this fact, \eqref{F1 error 1}, \eqref{F2 error 1} and \eqref{F3 error 1}, we have
	\begin{align}
	\|F(u^{l, n})\|_{H^{\delta}}\lesssim&\max
\left\{n^{1-2s+\delta},n^{-3s+\delta+\frac{2}{k}+2},n^{-2s+\delta+\frac{1}{k}+1},n^{-s+\delta-1}
\right\}	\lesssim n^{-r_s}.\label{error 2}
	\end{align}
	Then, for any $T>0$ and $t\in[0,T]$, by virtue of  the It\^{o} formula, we arrive at
	\begin{align*}
	\|E^{l, n}(t)\|^2_{H^\delta}\leq&\left|\left(-2\int_0^{t}h(t',u^{l,n}){\rm d}\W,E^{l,n}\right)_{H^\delta}\right|+\sum_{i=2}^4\int_0^t|J_i|{\rm d}t',
	\end{align*}
	where
	\begin{align*}
	J_2&=2\left(D^\delta T_{n,k},
	D^\delta E^{l, n}\right)_{L^2},\\
	J_3&=2\left(D^\delta F(u^{l,n}),D^\delta E^{l, n}\right)_{L^2},\\
	J_4&=\|h(t',u^{l,n})\|_{L_2(U,H^\delta)}^2.
	\end{align*}
	Taking the supremum with respect to $t\in[0,T]$, using the Burkholder-Davis-Gundy inequality and using \eqref{assumption G} and \eqref{error 2} yield
	\begin{align*}
	\E\sup_{t\in[0,T]}\left|\left(-2\int_0^{t}h(t',u^{l,n}){\rm d}\W,E^{l,n}\right)_{H^\delta}\right|
	\leq&2\E
	\left(\int_0^T\|E^{l, n}(t)\|_{H^\delta}^2
	\|F(u^{l, n})\|_{H^\delta}^2{\rm d}t\right)^{\frac12}\notag\\
	\leq& 2\E
	\left(\sup_{t\in[0,T]}\|E^{l, n}(t)\|_{H^\delta}^2
	\int_0^T\|F(u^{l, n})\|_{H^\delta}^2{\rm d}t\right)^{\frac12}\notag\\
	\leq& \frac12\E\sup_{t\in[0,T]}\|E^{l, n}(t)\|_{H^\delta}^2
	+CTn^{-2r_s}.
	\end{align*}
	By virtue of \eqref{error 1} and \eqref{error 2}, we obtain
	\begin{align*}
	\int_0^{T}\E|J_2|{\rm d}t
	\leq & C\int_0^{T}\E\left(\left\|T_{n,k}\right\|_{H^{\delta}}
	\|E^{l, n}(t)\|_{H^\delta}\right){\rm d}t\notag\\
	\leq & C\int_0^{T}\E
	\left\|T_{n,k}\right\|^2_{H^{\delta}}{\rm d}t
	+C\int_0^{T}\E\|E^{l, n}(t)\|_{H^\delta}^2{\rm d}t\notag\\
	\leq & CTn^{-2r_s}+C\int_0^{T}\E\|E^{l, n}(t)\|^2_{H^\delta}{\rm d}t,
	\end{align*}
	\begin{align*}
	\int_0^{T}\E|J_3|{\rm d}t
	\leq & C\int_0^{T}\E\left(\|F(u^{l, n})\|_{H^{\delta}}
	\|E^{l, n}(t)\|_{H^\delta}\right){\rm d}t\notag\\
	\leq & C\int_0^{T}\E\|F(u^{l, n})\|^2_{H^{\delta}}{\rm d}t
	+C\int_0^{T}\E\|E^{l, n}(t)\|^2_{H^\delta}{\rm d}t\notag\\
	\leq & CT n^{-2r_s}
	+C\int_0^{T}\E\|E^{l, n}(t)\|^2_{H^\delta}{\rm d}t,
	\end{align*}
	and by \eqref{assumption G},
	\begin{align*}
	\int_0^{T}\E|J_4|{\rm d}t
	\leq & C\int_0^{T}\E \|F(u^{l, n})\|^2_{H^{\delta}}{\rm d}t\leq  CT n^{-2r_s}.
	\end{align*}
	Collecting the above estimates, we arrive at
	\begin{align*}
	\E\sup_{t\in[0,T]}\|E^{l, n}(t)\|^2_{H^\delta}\leq CT n^{-2r_s}
	+C\int_0^{T}\E\sup_{t'\in[0,t]}\|E^{l, n}(t')\|^2_{H^\delta}{\rm d}t.
	\end{align*}
	Obviously, for each $n\ge1$ and $l\in\{-1,1\}$, $\E\sup_{t'\in[0,t]}\|E^{l, n}(t')\|^2_{H^\delta}$ is finite. Then we infer from the Gr\"{o}nwall inequality that
	\begin{align*}
	\E\sup_{t\in[0,T]}\|E^{l, n}(t)\|^2_{H^\delta}\leq
	Cn^{-2r_s},\ \ C=C(T).
	\end{align*}
	This is the desired result.
\end{proof}

\subsection{Construction of actual solutions}

Now we consider the following periodic boundary value problem with deterministic initial data $u^{l, n}(0,x)$,  i.e.,
\begin{equation}\label{appro periodic Cauchy problem}
\left\{\begin{aligned}
&{\rm d}u+\left[u^ku_x+F(u)\right]{\rm d}t
=h(t,u){\rm d}\W,\qquad t>0,\ x\in \T,\\
&u(0,x)=u^{l, n}(0,x), \qquad x\in \T.
\end{aligned} \right.
\end{equation}
Since $h$ satisfies \eqref{assumption G}, we see that \eqref{assumption 1 for h} and \eqref{assumption 2 for h} are also verified. Then Theorem \ref{Local pathwise solution} yields that for each $n\in\N$, \eqref{appro periodic Cauchy problem} has a uniqueness maximal pathwise solution $(u_{l,n},\tau^*_{l,n})$.

\subsection{Estimates on the errors}

\begin{Lemma}\label{v l n estimate lemma}
	Let $s>\frac{3}{2}$,
	$\frac12<\delta <\min \left\{s-1,\frac{3}{2}\right\}$ and $r_{s}>0$ be given as in Lemma \ref{error estimate lemma}. For any $R>1$ and $l$ satisfying \eqref{l n condition}, we define
	\begin{align}
	\tau^R_{l,n}:=\inf\left\{t\geq0:\|u_{l, n}\|_{H^s}> R\right\}.
	\label{u eta n stopping time}
	\end{align}
	Then for any $T>0$, when $n\gg 1$, we have that for $l$ satisfying \eqref{l n condition},
	\begin{align}
	\E\sup_{t\in[0,T\wedge\tau^R_{l,n}]}\|u^{l, n}-u_{l, n}\|^2_{H^{\delta}}\leq   Cn^{-2r_s},\ \ C=C(R,T),\label{v sigma norm}
	\end{align}
	and
	\begin{align}
	\E\sup_{t\in[0,T\wedge\tau^R_{l,n}]}\|u^{l, n}-u_{l, n}\|^2_{H^{2s-\delta}}\leq   Cn^{2s-2\sigma},\ \ C=C(R,T).\label{v 2s-sigma norm}
	\end{align}
\end{Lemma}
\begin{proof}
	We first notice that by Lemma \ref{cos sin approximate estimate}, for $l$ satisfying \eqref{l n condition},
	\begin{align}
	\|u^{l, n}(t)\|_{H^s}\lesssim 1,\ \ \forall\ t>0.\label{appro solution bounded}
	\end{align}
	Let $q=q^{l,n}=\sum_{i=0}^{k}\left(u^{l, n}\right)^{k-i}\left(u_{l, n}\right)^i$ and $v=v^{l,n}=u^{l, n}-u_{l, n}$. In view of \eqref{error definition}, \eqref{Error equation} and \eqref{appro periodic Cauchy problem}, we see that
	\begin{align*}
	v(t)+&\int_{0}^{t}\left[\frac{1}{k+1}\partial_x(qv)
	-F(u_{l, n})\right]{\rm d}t'=-\int_{0}^{t}h(t',u_{l, n})\ {\rm d}\W
	+\int_{0}^{t}T_{n,k}\ {\rm d}t'.
	\end{align*}
	For any $T>0$, we use the It\^{o} formula on $[0,T\wedge\tau^R_{l,n}]$, take the supremum over $t\in[0,T\wedge\tau^R_{l,n}]$ and use the Burkholder-Davis-Gundy inequality with noticing \eqref{assumption G} to find
	\begin{align*}
	\E\sup_{t\in[0,T\wedge\tau^R_{l,n}]}\|v(t)\|^2_{H^{\delta}}\leq C\E\left(\int_0^{T\wedge\tau^R_{l,n}}|K_1|^2\ {\rm d}t\right)^{\frac12}
	+\sum_{i=2}^5\E\int_0^{T\wedge\tau^R_{l,n}}|K_i|\ {\rm d}t, 
	\end{align*}
	where
	\begin{align*}
	K_1&=\|v\|_{H^\delta}\|F(u_{l,n})\|_{H^\delta},\\
	K_2&=2\left(D^\delta T_{n,k},
	D^\delta v\right)_{L^2},\\
	K_3&=-\frac{2}{k+1}\left(D^\delta \partial_x[qv],D^\delta v\right)_{L^2},\\
	K_4&=2\left(D^\delta F(u_{l,n}),D^\delta v\right)_{L^2},\\
	K_5&=\|h(t,u_{l,n})\|_{L_2(U,H^\delta)}^2.
	\end{align*}
	We can first infer from Lemma \ref{F lemma} that
	\begin{align*}
	\|F(u_{l,n})\|_{H^\delta}^2\lesssim &
	\left(\|F(u^{l,n})-F(u_{l,n})\|_{H^\delta}
	+\|F(u^{l,n})\|_{H^\delta}\right)^2\notag\\
	\lesssim &
	(\|u^{l,n}\|_{H^{s}}+\|u_{l,n}\|_{H^{s}})^2
	\|v\|^2_{H^{\delta}}
	+\|F(u^{l,n})\|^2_{H^\delta}.
	\end{align*}
	Therefore, applying Lemmas \ref{Taylor} and \ref{F lemma}, $H^{\delta}\hookrightarrow L^{\infty}$, integrating by part and \eqref{error 1}, we have
	\begin{align*}
	|K_1|^2\lesssim(\|u^{l, n}\|_{H^{s}}+\|u_{l, n}\|_{H^{s}})^2
	\|v\|^4_{H^{\delta}}
	+\|F(u^{l, n})\|^2_{H^\delta}\|v\|^2_{H^\delta},
	\end{align*}
	\begin{align*}
	|K_2|\lesssim \left\|T_{n,k}\right\|^2_{H^{\delta}}+\|v\|^2_{H^{\delta}}
	\lesssim n^{-2r_s}+\|v\|^2_{H^{\delta}},
	\end{align*}
	\begin{align*}
	|K_3|\lesssim &\|q\|_{H^s}\|v\|^2_{H^{\delta}}
	+\|q_x\|_{L^\infty}\|v\|^2_{H^{\delta}}
	\lesssim\|q\|_{H^s}\|v\|^2_{H^\delta},
	\end{align*}
	\begin{align*}
	|K_4|\lesssim(\|u^{l, n}\|_{H^{s}}+\|u_{l, n}\|_{H^{s}})
	\|v\|^2_{H^{\delta}}
	+\|F(u^{l, n})\|^2_{H^\delta}+\|v\|^2_{H^\delta},
	\end{align*}
	and
	\begin{align*}
	|K_5|\lesssim(\|u^{l, n}\|_{H^{s}}+\|u_{l, n}\|_{H^{s}})^2
	\|v\|^2_{H^{\delta}}
	+\|F(u^{l, n})\|^2_{H^\delta}.
	\end{align*}
	By virtue of Lemma \ref{F lemma}, \eqref{error 2}, \eqref{u eta n stopping time} and \eqref{appro solution bounded}, we have
	\begin{align*}
	&\hspace*{-1cm}C\E\left(\int_0^{T\wedge\tau^R_{l,n}}|K_1|^2\ {\rm d}t\right)^{\frac12}\notag\\
	\leq& C\E
	\left(\sup_{t\in[0,T\wedge\tau^R_{l,n}]}\|v\|_{H^{\delta}}^2
	\int_0^{T\wedge\tau^R_{l,n}}
	(\|u^{l, n}\|_{H^{s}}+\|u_{l, n}\|_{H^{s}})^2
	\|v\|^2_{H^{\delta}}\ {\rm d}t\right)^{\frac12}\notag\\
	&+C\E
	\left(\sup_{t\in[0,T\wedge\tau^R_{l,n}]}\|v\|_{H^{\delta}}^2
	\int_0^{T\wedge\tau^R_{l,n}}
	\|F(u^{l, n})\|^2_{H^\delta}\ {\rm d}t\right)^{\frac12}\notag\\
	\leq& \frac12\E\sup_{t\in[0,T\wedge\tau^R_{l,n}]}\|v\|_{H^{\delta}}^2
	+C_R\E\int_0^{T\wedge\tau^R_{l,n}}
	\|v(t)\|^2_{H^{\delta}}\ {\rm d}t
	+C\E\int_0^{T\wedge\tau^R_{l,n}}
	\|F(u^{l, n})\|^2_{H^\delta}\ {\rm d}t\notag\\
	\leq& \frac12\E\sup_{t\in[0,T\wedge\tau^R_{l,n}]}\|v\|_{H^{\delta}}^2
	+C_R\E\int_0^{T}
	\sup_{t'\in[0,t\wedge\tau^R_{l,n}]}\|v(t')\|^2_{H^\delta}\ {\rm d}t
	+CT n^{-2r_s},
	\end{align*}
	\begin{align*}
	\E\int_0^{T\wedge\tau^R_{l,n}}|K_2|+|K_4|+|K_5|\ {\rm d}t
	\leq  CTn^{-2r_s}
	+C_R\int_0^{T}\E
	\sup_{t'\in[0,t\wedge\tau^R_{l,n}]}\|v(t')\|^2_{H^{\delta}}\ {\rm d}t,
	\end{align*}
	and
	\begin{align*}
	\E\int_0^{T\wedge\tau^R_{l,n}}|K_3|\ {\rm d}t
	\leq  C_R\int_0^{T}\E
	\sup_{t'\in[0,t\wedge\tau^R_{l,n}]}\|v(t')\|^2_{H^{\delta}}\ {\rm d}t.
	\end{align*}
	Consequently,
	\begin{align*}
	\E\sup_{t\in[0,T\wedge\tau^R_{l,n}]}\|v(t)\|^2_{H^{\delta}}\leq   CTn^{-2r_s}
	+C_R\int_0^{T}\E
	\sup_{t'\in[0,t\wedge\tau^R_{l,n}]}\|v(t')\|^2_{H^{\delta}}\ {\rm d}t.
	\end{align*}
	Using the Gr\"{o}nwall inequality, we have
	\begin{align*}
	\E\sup_{t\in[0,T\wedge\tau^R_{l,n}]}\|v(t)\|^2_{H^{\delta}}\leq   Cn^{-2r_s},\ \ C=C(R,T),
	\end{align*}
	which is \eqref{v sigma norm}. For \eqref{v 2s-sigma norm}, we first notice 
	that $u_{l, n}$ is the unique solution to \eqref{appro periodic Cauchy problem} with $2s-\delta>3/2$. For each fixed $n\in \Z^{+}$,  we can go along the lines as we derive \eqref{u Hs estimate} with using \eqref{u eta n stopping time} to find 
	\begin{align}\label{actual solution 2s-delta estimate}
	\E\sup_{t\in[0,T\wedge\tau^R_{l,n}]}\|u_{l, n}(t)\|^2_{H^{2s-\delta}} 
	\leq &C\E\|u^{l, n}(0)\|^2_{H^{2s-\delta}}+ C_R\int_0^{T}
	\left(\E\sup_{t'\in[0,t\wedge\tau^R_{l,n}]}\|u_{l,n}(t')\|_{H^{2s-\delta}}^2\right){\rm d}t.
	\end{align}
	From \eqref{actual solution 2s-delta estimate}, we can use the Gr\"{o}nwall inequality and Lemma \ref{cos sin approximate estimate} to infer
	\begin{align*}
	\E\sup_{t\in[0,T\wedge\tau^R_{l,n}]}\|u_{l, n}(t)\|^2_{H^{2s-\delta}}
	\leq C\E\|u^{l, n}(0)\|^2_{H^{2s-\delta}}
	\leq C n^{2s-2\delta},\ \ C=C(R,T).
	\end{align*}
By Lemma \ref{cos sin approximate estimate} again, we have that for some $C=C(R,T)$ and $l$ satisfying \eqref{l n condition},
	\begin{align*}
	\E\sup_{t\in[0,T\wedge\tau^R_{l,n}]}\|v\|^2_{H^{2s-\delta}}
	\leq C\E\sup_{t\in[0,T\wedge\tau^R_{l,n}]}\|u_{l, n}\|^2_{H^{2s-\delta}}
	+C\E\sup_{t\in[0,T\wedge\tau^R_{l,n}]}\|u^{l, n}\|^2_{H^{2s-\delta}}
	\leq C n^{2s-2\delta},
	\end{align*}
	which is \eqref{v 2s-sigma norm}.
\end{proof}

\subsection{Final proof for Theorem \ref{weak instability}}

To begin with, we will show that if the exiting time of the zero solution is strongly stable, then  $\{u_{-1,n}\}$ and $\{u_{1,n}\}$ are two sequences of pathwise solutions such that
\eqref{Non--uniform dependence condition 1}--\eqref{Non--uniform dependence condition 4} are satisfied.

\begin{Lemma}\label{tau eta n tends to infty}
	If for some $R_0\gg1$, the $R_0$-exiting time of the zero solution to \eqref{SGCH problem} is strongly stable,  then for $l$ satisfying \eqref{l n condition} and $\tau^{R_0}_{l,n}$ given in \eqref{u eta n stopping time}, we have
	\begin{align}
	\lim_{n\rightarrow\infty} \tau^{R_0}_{l,n}=\infty\ \  \p-a.s.\label{tau eta n lower>0}
	\end{align}
\end{Lemma}

\begin{proof}
	We notice that for all $s'<s$,
	$
	\lim_{n\rightarrow\infty} \|u^{l, n}(0)\|_{H^{s'}}=\lim_{n\rightarrow\infty} \|u_{l, n}(0)-0\|_{H^{s'}}=0.
	$
	Since the unique solution with zero initial data to \eqref{SGCH problem} is zero  and the $R_0$-exiting time of the zero solution is $\infty$, we see that \eqref{tau eta n lower>0} holds true provided the $R_0$-exiting time of the zero solution to \eqref{SGCH problem} is strongly stable.
\end{proof}

\begin{proof}[Proof for Theorem \ref{weak instability}]
	For each $n>1$, for $l$ satisfying \eqref{l n condition} and for the fixed $R_0\gg1$, Lemma \ref{cos sin approximate estimate} and \eqref{u eta n stopping time} give us $\p\{\tau_{l, n}^{R_0}>0\}=1$ and Lemma \ref{tau eta n tends to infty} implies \eqref{Non--uniform dependence condition 1}. Besides, Theorem \ref{Local pathwise solution} and \eqref{u eta n stopping time} show that $u_{l, n}\in C([0,\tau_{l, n}^{R_0}];H^s)$ $\p-a.s.$ and
	\eqref{Non--uniform dependence condition 2} holds true.  By virtue of the interpolation inequality, we have
	\begin{align*}
	&\hspace*{-1cm}\E\sup_{t\in[0,T\wedge\tau_{l, n}^{R_0}]}\|u^{l, n}-u_{l, n}\|_{H^{s}}\\
	\leq& C 
	\left( \E\sup_{t\in[0,T\wedge\tau_{l, n}^{R_0}]}\|u^{l, n}-u_{l, n}\|_{H^{\delta}}\right)^{\frac12}
	\left( \E\sup_{t\in[0,T\wedge\tau_{l, n}^{R_0}]}\|u^{l, n}-u_{l, n}\|_{H^{2s-\delta}}\right)^{\frac12}\\
	\leq& C 
	\left( \E\sup_{t\in[0,T\wedge\tau_{l, n}^{R_0}]}\|u^{l, n}-u_{l, n}\|^2_{H^{\delta}}\right)^{\frac14}
	\left( \E\sup_{t\in[0,T\wedge\tau_{l, n}^{R_0}]}\|u^{l, n}-u_{l, n}\|^2_{H^{2s-\delta}}\right)^{\frac14}.
	\end{align*}
	For any $T>0$, combining Lemma \ref{v l n estimate lemma} and the above estimate yields
	\begin{align}
\E\sup_{t\in[0,T\wedge\tau_{l, n}^{R_0}]}
	\|u^{l, n}-u_{l, n}\|_{H^{s}}
	\lesssim  n^{-\frac{1}{4}\cdot 2r_{s}}\cdot n^{\frac{1}{4}\cdot(2s-2\delta)}=n^{r'_s},\label{difference r's}
	\end{align}
	for some $C=C(R_0,T)$ and $l$ satisfying \eqref{l n condition}.
Recalling $k\geq 2$ and \eqref{rs>0}, we have
\begin{equation*}
0>r'_s=-r_s\cdot \frac{1}{2}+(s-\delta)\cdot\frac{1}{2}=
\begin{cases}
\displaystyle-\frac{s}{2}+\frac{k+1}{2k}, & {\rm if}  \  \displaystyle\frac{3}{2}<s\leqslant\frac{2k+1}{k},\\[1ex]
\displaystyle-\frac{1}{2}, &  {\rm if} \ \displaystyle s>\frac{2k+1}{k}.
\end{cases}
\end{equation*}
	Therefore for $l$ satisfying \eqref{l n condition},
	\begin{equation}
	\lim_{n\rightarrow\infty}
	\E\sup_{t\in[0,T\wedge\tau_{l, n}^{R_0}]}\|u^{l, n}-u_{l, n}\|_{H^{s}}
	=0.\label{difference tends to 0}
	\end{equation}
		For odd $k$, \eqref{Non--uniform dependence condition 3} is given by
	\begin{align*} 
	\|u_{-1,n}(0)-u_{1,n}(0)\|_{H^{s}}=\|u^{-1,n}(0)-u^{1,n}(0)\|_{H^{s}}\lesssim n^{-1/k}\rightarrow 0, {~\rm as~}n\rightarrow \infty.
	\end{align*}
		For even $k$, \eqref{Non--uniform dependence condition 3} is given by
		\begin{align*} 
		\|u_{0,n}(0)-u_{1,n}(0)\|_{H^{s}}=\|u^{0,n}(0)-u^{1,n}(0)\|_{H^{s}}\lesssim n^{-1/k}\rightarrow 0, {~\rm as~}n\rightarrow \infty.
		\end{align*}
	When $k$ is odd, we use \eqref{difference tends to 0} to find that for any $T>0$,
	\begin{align}
	&\hspace*{-1cm}\liminf_{n\rightarrow \infty}
	\E\sup_{t\in[0,T\wedge\tau_{-1,n}^{R_0}\wedge\tau_{1,n}^{R_0}]}
	\|u_{-1,n}(t)-u_{1,n}(t)\|_{H^s}\notag\\
	\geq&
	\liminf_{n\rightarrow \infty}
	\E\sup_{t\in[0,T\wedge\tau_{-1,n}^{R_0}\wedge\tau_{1,n}^{R_0}]}
	\|u^{-1,n}(t)-u^{1,n}(t)\|_{H^s}\notag\\
	&-\lim_{n\rightarrow \infty}
	\E\sup_{t\in[0,T\wedge\tau_{-1,n}^{R_0}\wedge\tau_{1,n}^{R_0}]}
	\|u^{-1,n}(t)-u_{-1,n}(t)\|_{H^s}\notag\\
	&-\lim_{n\rightarrow \infty}\E\sup_{t\in[0,T\wedge\tau_{-1,n}^{R_0}\wedge\tau_{1,n}^{R_0}]}
	\|u^{1,n}(t)-u_{1,n}(t)\|_{H^s}\notag\\
	\gtrsim &\liminf_{n\rightarrow \infty}
	\E\sup_{t\in[0,T\wedge\tau_{-1,n}^{R_0}\wedge\tau_{1,n}^{R_0}]}
	\|u^{-1,n}(t)-u^{1,n}(t)\|_{H^s}\notag\\
	\gtrsim &\liminf_{n\rightarrow \infty}
	\E\sup_{t\in[0,T\wedge\tau_{-1,n}^{R_0}\wedge\tau_{1,n}^{R_0}]}
	\|-2n^{-1}+n^{-s}\cos(nx+t)-n^{-s}\cos(nx-t)\|_{H^s}\nonumber\\
	\gtrsim &\liminf_{n\rightarrow \infty}
	\E\sup_{t\in[0,T\wedge\tau_{-1,n}^{R_0}\wedge\tau_{1,n}^{R_0}]}
	\left(n^{-s}\|\sin(nx)\|_{H^s}|\sin t|-\|2n^{-1}\|_{H^s}\right)
	\gtrsim\sup_{t\in[0,T]}|\sin t|,\label{To prove Non--uni 4}
	\end{align}
	where we have used Fatou's lemma.
 By
	\begin{align*}
	&
	\left(\E\sup_{t\in[0,T\wedge\tau_{-1,n}^{R_0}\wedge\tau_{1,n}^{R_0}]}
	\|u_{-1,n}(t)-u_{1,n}(t)\|^2_{H^s}\right)^{\frac12}
	\geq
	\E\sup_{t\in[0,T\wedge\tau_{-1,n}^{R_0}\wedge\tau_{1,n}^{R_0}]}
	\|u_{-1,n}(t)-u_{1,n}(t)\|_{H^s}
	\end{align*}
	and \eqref{To prove Non--uni 4}, we obtain \eqref{Non--uniform dependence condition 4}.  Similarly, we can also prove \eqref{Non--uniform dependence condition 4} when $k$ is even.  The proof is therefore completed.
\end{proof}

\section{Global existence}

In this section, we study the global existence and the blow-up of solutions to \eqref{GCH linear noise}, and estimate the associated probabilities. Motivated by \cite{GlattHoltz-Vicol-2014-AP,Rockner-Zhu-Zhu-2014-SPTA,Tang-2018-SIMA}, we introduce the following Girsanov type transform
\begin{align}
v=\frac{1}{\beta(\omega,t)} u,\ \
\beta(\omega,t)={\rm e}^{\int_0^tb(t') {\rm d} W_{t'}-\int_0^t\frac{b^2(t')}{2} {\rm d}t'}.\label{transform}
\end{align}

\begin{Proposition}\label{pathwise solutions v}
	Let $s>3/2$ and $h(t,u)=b(t) u$ such that $b(t)$ satisfies Assumption \ref{Assumption-3}. Let $\s=\left(\Omega, \mathcal{F},\p,\{\mathcal{F}_t\}_{t\geq0}, W\right)$ be fixed in advance. If $u_0(\omega,x)$ is an $H^s$-valued $\mathcal{F}_0$-measurable random variable with $\E\|u_0\|^2_{H^s}<\infty$ and $(u,\tau^*)$ is the corresponding unique maximal pathwise solution to \eqref{GCH linear noise},
	then for $t\in[0,\tau^*)$, the process $v$ given in \eqref{transform} solves 
	\begin{equation} \label{periodic Cauchy problem transform}
	\left\{\begin{aligned}
	&v_t+\beta^k v^kv_x+\beta^k F(v)=0,\ \ t>0,\ x\in\T,\\
	&v(\omega,0,x)=u_0(\omega,x),\  x\in\T.
	\end{aligned} \right.
	\end{equation}
	Moreover, one has
	$v\in C\left([0,\tau^*);H^s\right)\bigcap C^1([0,\tau^*) ;H^{s-1})$ $\p-a.s.$ and, if $s>3$, then
	\begin{equation}
	\p\{\|v(t)\|_{H^1}=\|u_0\|_{H^1}\}=1.\label{H1 conservation}
	\end{equation}
\end{Proposition}
\begin{proof}
	Since $b(t)$ satisfies Assumption \ref{Assumption-3},  $h(t,u)=b(t) u$ satisfies Assumption \ref{Assumption-1}. Consequently, Theorem \ref{Local pathwise solution} implies  that \eqref{GCH linear noise} has a unique maximal pathwise solution $(u,\tau^*)$.
	Direct computation shows
	\begin{align}
	{\rm d}v
	=&\left(-\beta^k v^kv_x-\beta^k F(v)\right){\rm d}t.\label{v equation}
	\end{align}
	Since $v(0)=u_0(\omega,x)$, we see that $v$  satisfies \eqref{periodic Cauchy problem transform}. Moreover,  Theorem \ref{Local pathwise solution} implies $u\in C\left([0,\tau^*);H^s\right)$ $\p-a.s.$, so $v\in C([0,\tau^*); H^{s})$ $\p-a.s.$ Besides, from Lemma \ref{F lemma} and \eqref{periodic Cauchy problem transform}$_1$, we see that for a.e. $\omega\in\Omega$, $v_t=-\beta^k v^kv_x-\beta^k F(v)\in C([0,\tau^*); H^{s-1})$. Hence $v\in C^1\left([0, \tau^*);H^{s-1}\right)$ $\p-a.s.$ 

	Notice that if $s>3$, \eqref{periodic Cauchy problem transform}$_1$ is equivalent to
	\begin{equation}\label{v equation full}
	v_t-v_{xxt}+(k+2)\beta^k v^kv_x=(k+1)\beta^kv^{k-1}v_xv_{xx}+\beta^kv^kv_{xxx},\ \ k\geq1.
	\end{equation}
	Multiplying both sides of the above equation by $v$ and then integrating the resulting equation on $x\in\T$ with noticing that 
	$(k+1)v^{k}v_xv_{xx}+v^{k+1}v_{xxx}=\partial_x\left(v^{k+1}v_{xx}\right),$ 
   we see that for a.e. $\omega\in\Omega$ and for all $t>0$
	$$\frac{\rm d}{{\rm d}t}\int_{\T}\left(v^2+v_x^2\right){\rm d}x=0,$$ which implies \eqref{H1 conservation}.
\end{proof}

\subsection{Global existence: Case 1}

Now we prove the first global existence result.

\begin{proof}[Proof for Theorem \ref{Decay result}]
	To begin with, we apply the operator $D^s$ to \eqref{v equation}, multiply both sides of the resulting equation by $D^sv$ and integrate over $\T$ to obtain that for  a.e. $\omega\in\Omega$,
	\begin{align*}
	\frac{1}{2}\frac{\rm d}{{\rm d}t}\|v(t)\|^2_{H^s}
	=	-\beta^k(\omega,t)	\int_{\T}D^sv\cdot D^s\left[v^kv_x\right]{\rm d}x
	-\beta^k(\omega,t)\int_{\T}D^sv\cdot D^sF(v) {\rm d}x.
	\end{align*}
	Using Lemmas \ref{Kato-Ponce commutator estimate} and \ref{F lemma}, 
	we conclude that there is a $C=C(s)>1$ such that for a.e. $\omega\in\Omega$, 
	\begin{align*}
	\frac{\rm d}{{\rm d}t}\|v(t)\|^2_{H^s}
	\leq C\beta^k(t)\ \|v\|^k_{W^{1,\infty}}\|v\|_{H^s}^2,
	\end{align*}
	where $\beta$ is given in \eqref{transform}.
	Then
	$w={\rm e}^{-\int_0^tb(t') {\rm d} W_{t'}}u={\rm e}^{-\int_0^t\frac{b^2(t')}{2} {\rm d}t'}v$ satisfies
	\begin{align*}
	\frac {\rm d}{ {\rm d}t}\|w(t)\|_{H^s}+\frac{b^2(t)}{2}\|w(t)\|_{H^s}
	\leq C\alpha^k(\omega,t) \|w(t)\|^k_{W^{1,\infty}}\|w(t)\|_{H^s},\ \
	\alpha(\omega,t)={\rm e}^{\int_0^tb(t') {\rm d} W_{t'}}. 
	\end{align*}
	Let $R>1$ and $\lambda_1>2$. Assume 
	$\|u_0\|_{H^s}<\frac{1}{RK}\left(\frac{b_*}{C\lambda_1}\right)^{1/k} $ almost surely. Define
	\begin{align}
	\tau_{1}(\omega)=\inf\left\{t>0:\alpha^k(\omega,t) \|w\|^k_{W^{1,\infty}}
	=\|u\|^k_{W^{1,\infty}}>\frac{b^2(t)}{C\lambda_1 }\right\}.\label{global time tau}
	\end{align}
    Notice that $\|u(0)\|^k_{W^{1,\infty}}\leq K^{k}\|u(0)\|^k_{H^s}<\frac{b_*}{C\lambda_1}$. Therefore we have
	$
	\p\{\tau_{1}>0\}=1,
	$
	and for $t\in[0,\tau_{1})$,
	\begin{align*}
	\frac {\rm d}{ {\rm d}t}\|w(t)\|_{H^s}+\frac{(\lambda_1-2)b^2(t)}{2\lambda_1}\|w(t)\|_{H^s}
	\leq 0.
	\end{align*}
	The above inequality and $w={\rm e}^{-\int_0^tb(t'){\rm d} W_{t'}}u$ imply that for a.e. $\omega\in\Omega$, for any $\lambda_2>\frac{2\lambda_1}{\lambda_1-2}$ and for $t\in[0,\tau_{1})$,
	\begin{align}
	\|u(t)\|_{H^s}
	\leq& \|w_0\|_{H^s}{\rm e}^{\int_0^tb(t')  {\rm d} W_{t'}-\int_0^t\frac{(\lambda_1-2)b^2(t')}{2\lambda_1} {\rm d}t'}\notag\\
	=&\|u_0\|_{H^s}{\rm e}^{\int_0^tb(t') {\rm d} W_{t'}-\int_0^t\frac{b^2(t')}{\lambda_2} {\rm d}t'}{\rm e}^{-\frac{\left((\lambda_1-2)\lambda_2-2\lambda_1\right)}{2\lambda_1\lambda2}\int_0^tb^2(t') {\rm d}t'}
	\label{Extracting some damping}
	\end{align}
	Define the stopping time
	\begin{equation}\label{tau 2 Girsanov}
	\tau_{2}=\tau_2(\omega)
	=\inf\left\{t>0:{\rm e}^{\int_0^tb(t') {\rm d} W_{t'}-\int_0^t\frac{b^2(t')}{\lambda_2} {\rm d}t'}>R\right\}.
	\end{equation}
	Notice that $\p\{\tau_{2}>0\}=1$. From \eqref{Extracting some damping}, we have that almost surely
	\begin{align}
	\|u(t)\|_{H^s}<& \frac{1}{RK}\left(\frac{b_*}{C\lambda_1}\right)^{1/k} \times R\times
	{\rm e}^{-\frac{\left((\lambda_1-2)\lambda_2-2\lambda_1\right)}{2\lambda_1\lambda2}\int_0^tb^2(t') {\rm d}t'}\notag\\
	=&\frac{1}{K}\left(\frac{b_*}{C\lambda_1}\right)^{1/k} 
	{\rm e}^{-\frac{\left((\lambda_1-2)\lambda_2-2\lambda_1\right)}{2\lambda_1\lambda2}\int_0^tb^2(t') {\rm d}t'}
	\leq\frac{1}{K}\left(\frac{b_*}{C\lambda_1}\right)^{1/k},\ \ t\in[0,\tau_{1}\wedge \tau_{2}).
	\label{u energy estimate R theta}
	\end{align}
	Combining \eqref{u energy estimate R theta} and \eqref{global time tau}, we find that
	\begin{equation}\label{tau 1 > tau 2}
	\p\{\tau_{1}\geq\tau_{2}\}=1.
	\end{equation}
	Therefore it follows from \eqref{u energy estimate R theta} that
	$$\p\left\{
	\|u(t)\|_{H^s}<\frac{1}{K}\left(\frac{b_*}{C\lambda_1}\right)^{1/k} 
	{\rm e}^{-\frac{\left((\lambda_1-2)\lambda_2-2\lambda_1\right)}{2\lambda_1\lambda2}\int_0^tb^2(t') {\rm d}t'}	\ {\rm\ for\ all}\ t>0
	\right\}\geq
	\p\{\tau_{2}=+\infty\}.$$
	We apply \ref{exit time eta} in Lemma \ref{eta Lemma} to find that
	\begin{equation*}
	\p\{\tau_{2}=+\infty\}>1-\left(\frac{1}{R}\right)^{2/\lambda_2},
	\end{equation*}
	which completes the proof.
\end{proof}

\subsection{Global existence: Case 2} Now we prove Theorem \ref{Global existence result}.
Let $\beta(\omega,t)$ be given in \eqref{transform}.
From Proposition \ref{pathwise solutions v}, we see that for a.e. $\omega\in\Omega$, $v(\omega,t,x)$  solves \eqref{periodic Cauchy problem transform} on $[0,\tau^*)$. Moreover, since $H^s\hookrightarrow C^2$ for $s>3/2$, we have $v,v_x\in C^1\left([0, \tau^*)\times\T\right)$. Then for a.e. $\omega\in\Omega$, for any $x\in\T$, the problem
\begin{equation} \label{particle line}
\left\{\begin{aligned}
&\frac{dq(\omega,t,x)}{dt}=\beta^k(\omega,t)v^k(\omega,t,q(\omega,t,x)),\ \ \ \ t\in[0,\tau^*),\\
&q(\omega,0,x)=x,\ \ \ x\in \T,
\end{aligned} \right.
\end{equation}
has a unique solution $q(\omega,t,x)$ such that $q(\omega,t,x)\in C^1([0,\tau^*)\times \T)$ almost surely.
Moreover, differentiating \eqref{particle line} with respect to $x$ yields that for a.e. $\omega\in\Omega$,
\begin{equation*}
\left\{\begin{aligned}
&\frac{dq_x}{dt}=k\beta^k(\omega,t)v^{k-1}v_xq_x,\ \ \ \ t\in[0,\tau^*),\\
&q_x(\omega,0,x)=1,\ \ \ x\in \T.
\end{aligned} \right.
\end{equation*}
For a.e. $\omega\in\Omega$, we solve the above equation to obtain
$$q_x(\omega,t,x)=\exp{\left(\int_0^tk\beta^k(\omega,t')v^{k-1}v_x(\omega,t',q(\omega,t',x))\ {\rm d}t'\right)}.$$ Thus for all $(t,x)\in[0, \tau^*)\times \T$, we find $q_x>0$ almost surely.  
\begin{Lemma}\label{same sign with initial data}
	Let $s>3$, $V_0(\omega,x)=(1-\partial_{xx}^2)u_0(\omega,x)$ and $V(\omega,t,x)=v(\omega,t,x)-v_{xx}(\omega,t,x)$, where $v(\omega,t,x)$ solves \eqref{periodic Cauchy problem transform} on $[0,\tau^*)$ $\p-a.s.$ Then for all $(t,x)\in[0, \tau^*)\times \T$,
	\begin{align*}
	{\rm sign}(v)={\rm sign}(V)={\rm sign}&(V_0) \ \  \p-a.s.
	\end{align*}
\end{Lemma}

\begin{proof}
	We first notice that \eqref{periodic Cauchy problem transform}$_1$ is equivalent to \eqref{v equation full}. Therefore for a.e. $\omega\in \Omega$,  $V=v-v_{xx}$ satisfies
	\begin{equation*}
	V_{t}+\beta^k v^kV_{x}+(k+1)\beta^k V v^{k-1}v_{x}=0.
	\end{equation*}
	Thus for a.e. $\omega\in \Omega$, we have
	\begin{align*}
	\frac{\rm d}{{\rm d}t}\left[V(\omega,t,q(\omega,t,x))q_x^{\frac{k+1}{k}}(\omega,t,x)\right]
	=&V_tq^{\frac{k+1}{k}}_x+V_xq_tq^{\frac{k+1}{k}}_x+\frac{k+1}{k}Vq_x^{\frac{k+1}{k}-1}q_{xt}\\
	=&q_x^{\frac{k+1}{k}}\left[V_t+V_x\left(\beta^k v^k\right)+\frac{k+1}{k}Vq_x^{-1} \left(k\beta^k v^{k-1}v_xq_x\right)\right]\\
	=&q_x^{\frac{k+1}{k}}\left[V_t+\beta^k v^k V_x+(k+1)\beta^k Vv^{k-1} v_x\right]=0.
	\end{align*}
	Notice that $q_x(\omega,0,x)=1$ and $q_x>0$ almost surely. Then we have that
	${\rm sign}(V)={\rm sign}(V_0)$ $\p-a.s.$ Since $v=G_{\T}*V$ with $G_{\T}>0$ given in \eqref{Helmboltz operator}, we have ${\rm sign}(v)={\rm sign}(V)$ $\p-a.s.$
\end{proof}

\begin{Lemma}\label{vx bounded lemma}
	Let all the conditions as in the statement of Proposition \ref{pathwise solutions v} hold true. If
	\begin{align*}
	\p\{V_0(\omega,x)>0,\  \forall\ x\in\T\}=p,\ \
	\p\{V_0(\omega,x)<0,\  \forall\ x\in\T\}=q,
	\end{align*}
	for some $p,q\in[0,1]$, then the maximal pathwise solution $u$ to \eqref{GCH linear noise} satisfies
	\begin{equation*}
	\p\left\{\|u_x(\omega,t)\|_{L^{\infty}}\leq \|u(\omega,t)\|_{L^{\infty}}\lesssim\beta(\omega,t)\|u_0\|_{H^1},\  \forall\ t\in[0,\tau^*)\right\}\geq p+q.
	\end{equation*}
\end{Lemma}
\begin{proof}
	It is easy to see that for a.e. $\omega\in \Omega$ and for all $(t,x)\in[0, \tau^*)\times \T$,
	\begin{align}
	\left[v+v_x\right](\omega,t,x)=&\frac{1}{2\sinh(\pi)}
	\int^{2\pi}_0{\rm e}^{(x-y-2\pi\left[\frac{x-y}{2\pi}\right]-\pi)}V(\omega,t,y)\ {\rm d}y,\label{u+ux}\\
	\left[v-v_x\right](\omega,t,x)
	=&\frac{1}{2\sinh(\pi)}
	\int^{2\pi}_0{\rm e}^{(y-x+2\pi\left[\frac{x-y}{2\pi}\right]+\pi)}V(\omega,t,y)\ {\rm d}y.\label{u-ux}
	\end{align}
	Then one can employ \eqref{u+ux}, \eqref{u-ux} and Lemma \ref{same sign with initial data} to obtain that for a.e. $\omega\in \Omega$ and for all $(t,x)\in[0, \tau^*)\times \T$,
	\begin{equation}\label{ux u set}
	\left\{\begin{aligned}
	-v(\omega,t,x)\leq v_x(\omega,t,x)\leq v(\omega,t,x),\ \ & {\rm if}\ \ V_0(\omega,x)=(1-\partial_{xx}^2)u_0(\omega,x)>0, \\
	v(\omega,t,x)\leq v_x(\omega,t,x)\leq -v(\omega,t,x),\ \ & {\rm if}\ \ V_0(\omega,x)=(1-\partial_{xx}^2)u_0(\omega,x)<0.
	\end{aligned} \right.
	\end{equation}
	Notice that
	\begin{align}\label{non intersection}
	\left\{V_0(\omega,x)>0\right\}\bigcap \left\{V_0(\omega,x)<0\right\}=\emptyset.
	\end{align}
	Combining \eqref{ux u set} and \eqref{non intersection} yields
	\begin{equation}
	\p\left\{|v_x(\omega,t,x)|\leq |v(\omega,t,x)|,\  \forall\ (t,x)\in[0, \tau^*)\times \T\right\}\geq p+q.\label{v<vx probability}
	\end{equation}
	In view of $H^1\hookrightarrow L^{\infty}$, \eqref{H1 conservation} and \eqref{v<vx probability}, we arrive at
	\begin{equation*}
	\p\left\{\|v_x(\omega,t)\|_{L^{\infty}}
	\leq \|v(\omega,t)\|_{L^{\infty}}
	\lesssim \|v(\omega,t)\|_{H^1}
	=\|u_0\|_{H^1},\ \forall\ t\in[0,\tau^*)\right\}\geq p+q.
	\end{equation*}
	Via \eqref{transform}, we obtain the desired estimate.
\end{proof}

\begin{proof}[Proof for Theorem \ref{Global existence result}]
	
	Let $(u,\tau^*)$ be the maximal pathwise solution to \eqref{GCH linear noise}. 
	Then Lemma \ref{vx bounded lemma} implies that 
	\begin{equation*}
	\p\left\{\|u\|_{W^{1,\infty}}\lesssim 2\beta(\omega,t)\|u_0\|_{H^1},\  \forall\ t\in[0,\tau^*)\right\}\geq p+q.
	\end{equation*}
	It follows from \ref{eta->0} in Lemma \ref{eta Lemma} that $\sup_{t>0}\beta(\omega,t)<\infty$ $\p-a.s.$ Then we can infer from \eqref{Blow-up criteria common}   that $\p\{\tau^*=\infty\}\geq p+q$.
	That is to say,
	$
	\p\left\{u\ {\rm exists\ globally}\right\}\geq p+q
	$.
\end{proof}

\section*{Acknowledgement}
The authors would like to express their great gratitude to the anonymous referees for their valuable suggestions, which have led to a meaningful improvement of this paper.

\appendix\section{Auxiliary results}\label{appendix}

In this section we formulate and prove some estimates employed in above proofs.  We first recall the Friedrichs mollifier $J_{\e}$ defined as $$J_{\varepsilon}f(x)=j_{\varepsilon}\ast f(x).$$ Here
$\ast$ stands for the convolution, $j_{\varepsilon}(x)=\sum_{k\in{\Z}}\widehat{j}(\varepsilon k){\rm e}^{{\rm i}xk}$ and $j(x)$ is a Schwartz function satisfying $0\leq\widehat{j}(\xi)\leq1$ for all the $\xi\in \R$ and $\widehat{j}(\xi)=1$ for any $\xi\in[-1,1]$. We also need another mollifier $T_\e$ on $\mathbb{T}^d$ with $d\geq1$ as
$$T_\e f(x):=(1-\e^2 \Delta)^{-1}f(x)= \sum_{k\in\Z^d} \left(1+\e^2 |k|^2\right)^{-1}  \widehat{f}(k)\, {\rm e}^{{\rm i}x\cdot k},\ \ \e\in(0,1).$$
From the construction, we see that $J_\e$ and $T_\e$ enjoy the following estimates, see \cite{Tang-2018-SIMA, Tang-2020-Arxiv} for example,
\begin{align}
[D^s,J_{\varepsilon}]=[D^s,T_{\varepsilon}]=0,\label{mollifier property 3}
\end{align}
\begin{align}
(J_{\varepsilon}f, g)_{L^2}=(f, J_{\varepsilon}g)_{L^2},\ (T_{\varepsilon}f, g)_{L^2}&=(f, T_{\varepsilon}g)_{L^2},\label{mollifier property 4}
\end{align}
\begin{align}
\|J_{\varepsilon}u\|_{H^s},\|T_{\varepsilon}u\|_{H^s}&\leq \|u\|_{H^s}.\label{mollifier property 5}
\end{align}

\begin{Lemma}[\cite{Tang-2020-Arxiv}]\label{Te commutator} 
	Let $d\geq1$ and $f,g:\T^d\rightarrow\R^d$ such that $g\in W^{1,\infty}$ and $f\in L^2$. Then for some $C>0$,
	\begin{align*}
		\|[T_{\varepsilon}, (g\cdot \nabla)]f\|_{L^2}
		\leq C\|\nabla g\|_{L^\infty}\|f\|_{L^2}.
	\end{align*}
\end{Lemma}

\begin{Lemma}[\cite{Kato-Ponce-1988-CPAM}]\label{Kato-Ponce commutator estimate}
	If $f\in H^s\bigcap W^{1,\infty},\ g\in H^{s-1}\bigcap L^{\infty}$ for $s>0$, then
	$$
	\|\left[D^s,f\right]g\|_{L^2}\leq C_s(\|D^sf\|_{L^2}\|g\|_{L^{\infty}}+\|\partial_xf\|_{L^{\infty}}\|D^{s-1}g\|_{L^2}).
	$$
\end{Lemma}
\begin{Lemma}[\cite{Kato-Ponce-1988-CPAM}]
	\label{Moser estimate}
	Let $s>0$, for all $f,g \in H^s\bigcap L^{\infty}$,
	$$\|fg\|_{H^s}\leq C_s(\|f\|_{H^s}\|g\|_{L^{\infty}}+\|f\|_{L^{\infty}}\|g\|_{H^s}).$$
\end{Lemma}

\begin{Lemma}[Proposition 4.2, \cite{Taylor-2003-PAMS}]\label{Taylor}
	If $\rho>3/2$ and $0\leq \eta+1\leq \rho$, then for some $c>0$,
	$$\|[D^{\eta}\partial_x, f]v\|_{L^2}\leq c\|f\|_{H^{\rho}}\|v\|_{H^{\eta}},
	\ \ \ \ \forall\ f\in H^{\rho}, v\in H^{\eta}.$$
\end{Lemma}

\begin{Lemma}\label{F  lemma}
	For the $F(\cdot)$ defined in \eqref{F decomposition} with $F_3(u)$ disappearing for $k=1$, we have that for all $k\ge1$,
	\begin{align*}
	\|F(v)\|_{H^s}&
	\lesssim \|v\|_{W^{1,\infty}}^k\|v\|_{H^s},
	\ \ s>3/2,\\
	\|F(u)-F(v)\|_{H^s}&\lesssim\left(\|u\|_{H^s}+\|v\|_{H^s}\right)^k\|u-v\|_{H^s},
	\ \ s>3/2,\\
	\|F(u)-F(v)\|_{H^s}&\lesssim\left(\|u\|_{H^{s+1}}+\|v\|_{H^{s+1}}\right)^k\|u-v\|_{H^s},
	\ \  1/2<s\leq3/2
	\end{align*}
\end{Lemma}

\begin{proof}
	Since $s>3/2$, the desired result follows from Lemma \ref{Moser estimate} immediately. We omit the details and we refer to \cite{Tang-Zhao-Liu-2014-AA} for the CH case.
\end{proof}

\begin{Lemma}[\cite{Himonas-Kenig-Misiolek-2010-CPDE,Tang-Zhao-Liu-2014-AA}]\label{cos sin approximate estimate}
	Let $\sigma, \alpha\in\R$. If $\lambda\in\Z^+$ and $\lambda\gg 1$, then
	\begin{align*}
	\|\sin(\lambda x-\alpha)\|_{H^\sigma}
	=\|\cos(\lambda x-\alpha)\|_{H^\sigma}&\approx \lambda^{\sigma}.
	\end{align*}
\end{Lemma}

\begin{Lemma}\label{eta Lemma}
	Let Assumption \ref{Assumption-3} hold true and let $a(t)\in C([0,\infty))$ be a bounded function. For
	\begin{equation*}
	X={\rm e}^{\int_0^tb(t') {\rm d} W_{t'}+\int_0^ta(t')-\frac{b^2(t')}{2} {\rm d}t'},
	\end{equation*}
	we have the following properties:
	
	\begin{enumerate}[label={ (\roman*)}]
		
		\item\label{eta->0}	Let $\phi(t):=\int_0^tb^2(t') {\rm d}t'$ and $\phi^{-1}(t)$ be the inverse function of $\phi$.
		If 	\begin{equation}
		\limsup_{t \to \infty} \frac{1}{\sqrt{2 t \log \log t}} \biggl(\int_0^{\phi^{-1}(t)} a(t'){\rm d}t' 
		- \frac{t}{2}\biggr) < -1 ,\label{theta beta condition I}
		\end{equation}
		then
		\begin{equation}\label{eta tends 0}
		\lim_{t \to \infty}X(t)=0 \ \ \p-a.s.
		\end{equation}
		If
		\begin{equation}
		\liminf_{t \to \infty} \frac{1}{\sqrt{2 t \log \log t}} \biggl(\int_0^{\phi^{-1}(t)} a(t')  {\rm d}t' - \frac{t}{2}\biggr)>1 ,\label{theta beta condition II}
		\end{equation}
		then
		\begin{equation}\label{eta tends infty}
		\lim_{t \to \infty}X(t)=+\infty \ \ \p-a.s.
		\end{equation}

		\item\label{exit time eta} Let  $a(t)=\lambda b^2(t)$ with $\lambda<\frac12$ and $\tau_{R}=\inf \{t \geq 0 : X(t)>R\}$ with $R>1$, then
		\begin{equation}\label{Estimate on exist time}
		\mathbb{P}\left(\tau_{R}=\infty\right) \geq 1-\left(\frac{1}{R}\right)^{1-2\lambda}.
		\end{equation}
	\end{enumerate}
\end{Lemma}

\begin{proof}
	Let $Q(t)=\int_0^tb(t') {\rm d} W_{t'}+\int_0^ta(t')-\frac{b^2(t')}{2} {\rm d}t'$. Since  $\phi:[0,\infty)\ni t\mapsto\phi(t)\in[0,\infty)$ is one-to-one (since $\phi$ is strictly increasing) and onto (by Assumption \ref{Assumption-3}), $\phi^{-1}(t)$ is well defined and $\lim_{t\to\infty}Q(t)=\lim_{t\to\infty}Q(\phi^{-1}(t))$ $\p-a.s.$ 
	Direct computation (see \cite[Exercise 7.7  and Theorem 8.2]{Baldi-2017-book} for example) shows that $$B_t:=\int_0^{\phi^{-1}(t)}b(t'){\rm d} W_{t'}$$ is itself a Brownian motion. Then we use the  law of the iterated logarithm to get  $$\limsup_{t\to\infty}\frac{B_t}{\sqrt{2 t \log \log t}}=1,\ \  \liminf_{t\to\infty}\frac{B_t}{\sqrt{2 t \log \log t}}=-1 \ \ \p-a.s.$$
	Therefore \eqref{theta beta condition I} implies
	\begin{align*}
	\lim_{t\to\infty}Q(\phi^{-1}(t))\leq
	&\limsup_{t\rightarrow\infty}\sqrt{2 t \log \log t}\left(\frac{B_t}{\sqrt{2 t \log \log t}}+\frac{\int_0^{\phi^{-1}(t)}a(t') {\rm d}t'-\frac{t}{2}}{\sqrt{2 t \log \log t}}\right)
	=-\infty\ \ \p-a.s.,
	\end{align*}
	which leads to
	\begin{equation*}
	\lim_{t \rightarrow \infty}X(t)={\rm e}^{\lim_{t\to\infty}Q(\phi^{-1}(t))}=0 \ \ \p-a.s.
	\end{equation*}
	Hence we obtain 	\eqref{eta tends 0}. If
	\eqref{theta beta condition II} holds true, then
	\begin{align*}
	\lim_{t\to\infty}Q(\phi^{-1}(t))\geq
	&\liminf_{t\rightarrow\infty}\sqrt{2 t \log \log t}\left(\frac{B_t}{\sqrt{2 t \log \log t}}+\frac{\int_0^{\phi^{-1}(t)}a(t') {\rm d}t'-\frac{t}{2}}{\sqrt{2 t \log \log t}}\right)
	=+\infty\ \ \p-a.s.,
	\end{align*}
	which gives \eqref{eta tends infty}. Now we consider \eqref{Estimate on exist time} and exploit  \cite{GlattHoltz-Vicol-2014-AP}. It is easy to see that when 
	$a(t)=\lambda b^2(t)$, $X$
	is the unique global solution to the problem
	$$	 {\rm d}X=\lambda b^2(t)X {\rm d}t+b(t)X {\rm d}W,\ \ X(0)=1.$$
	We apply the It\^{o} formula to \(X^{p}\)
	with $p>0$ to obtain that
	$${\rm d }X^{p}=\left( p\lambda b^2(t)+\frac{b^{2}(t) p(p-1)}{2}\right) X^{p}  {\rm d} t+p b(t) X^{p} {\rm d }W.$$
	Particularly, if $p=1-2\lambda$, then
	$${\rm d }X^{1-2\lambda}=\left(1-2\lambda\right)b(t) X^{1-2\lambda} {\rm d }W, $$
	which means  $\mathbb{E} X^{1-2\lambda}\left(t \wedge \tau_{R}\right)=1.$
	Hence by the definition of $\tau_R$, continuity of measures and the Chebyshev inequality, we have
	\begin{align*}
	\mathbb{P}\left\{\tau_{R}=\infty\right\}
	=&\lim _{N \rightarrow \infty} \mathbb{P}\left(\tau_{R}>N\right)\\
	=&\lim _{N \rightarrow \infty} \mathbb{P}\left(X^{1-2\lambda}\left(N \wedge \tau_{R}\right)<R^{1-2\lambda}\right) \\
	\geq& \lim _{N \rightarrow \infty}\left(1-\frac{\mathbb{E} X^{1-2\lambda}\left(N \wedge \tau_{R}\right)}{R^{1-2\lambda}}\right)=1-\frac{1}{R^{1-2\lambda}},
	\end{align*}
	which is \eqref{Estimate on exist time}.
\end{proof}

%
%

\end{document}